\documentclass{amsart}
\pdfoutput=1

\usepackage{graphicx}
\usepackage{url}
\usepackage{microtype}
\usepackage{tikz}
\usepackage[pdftex,colorlinks,citecolor=black,linkcolor=black,urlcolor=black,bookmarks=false]{hyperref}
\def\MR#1{\href{http://www.ams.org/mathscinet-getitem?mr=#1}{MR#1}}
\def\arXiv#1{arXiv:\href{http://arXiv.org/abs/#1}{#1}}
\usepackage{doi}
\usepackage{booktabs}

\theoremstyle{plain}
\newtheorem{theorem}{Theorem}

\newtheorem{lemma}[theorem]{Lemma}
\newtheorem{proposition}[theorem]{Proposition}

\numberwithin{theorem}{section}

\numberwithin{equation}{section}

\newcommand{\R}{{\mathbb R}}

\newcommand{\Z}{{\mathbb Z}}
\newcommand{\C}{{\mathbb C}}

\newcommand{\vol}[1]{\mathop{\textup{vol}}\mathopen{}#1\mathclose{}}
\newcommand{\SL}{\mathop{\textup{SL}}}
\newcommand{\cP}{\mathcal{P}}
\newcommand{\cM}{\mathcal{M}}
\newcommand{\UHP}{\mathfrak{h}}

\renewcommand{\Im}{\operatorname{Im}}
\renewcommand{\Re}{\operatorname{Re}}

\title{A conceptual breakthrough in sphere packing}

\author{Henry Cohn}

\thanks{Henry Cohn is principal researcher at
Microsoft Research New England and adjunct professor
of mathematics at the Massachusetts Institute of Technology.
His email address is \texttt{cohn@microsoft.com}.}

\date{}

\begin{document}

\maketitle

On March 14, 2016, the world of mathematics received an extraordinary Pi Day
surprise when Maryna Viazovska posted to the arXiv a solution of the sphere
packing problem in eight dimensions \cite{Viazovska2016}.  Her proof shows
that the $E_8$ root lattice is the densest sphere packing in eight
dimensions, via a beautiful and conceptually simple argument.  Sphere packing
is notorious for complicated proofs of intuitively obvious facts, as well as
hopelessly difficult unsolved problems, so it's wonderful to see a relatively
simple proof of a deep theorem in sphere packing.  No proof of optimality had
been known for any dimension above three, and Viazovska's paper does not even
address four through seven dimensions.  Instead, it relies on remarkable
properties of the $E_8$ lattice. Her proof is thus a notable contribution to
the story of $E_8$, and more generally the story of exceptional structures in
mathematics.

One measure of the complexity of a proof is how long it takes the community
to digest it.  By this standard, Viazovska's proof is remarkably simple.  It
was understood by a number of people within a few days of her arXiv posting,
and within a week it led to further progress: Abhinav Kumar, Stephen D.\
Miller, Danylo Radchenko, and I worked with Viazovska to adapt her methods to
prove that the Leech lattice is an optimal sphere packing in twenty-four
dimensions \cite{CKMRV24}.  This is the only other case above three
dimensions in which the sphere packing problem has been solved.

The new ingredient in Viazovska's proof is a certain special function, which
enforces the optimality of $E_8$ via the Poisson summation formula.  The
existence of such a function had been conjectured by Cohn and Elkies in 2003,
but what sort of function it might be remained mysterious despite
considerable effort. Viazovska constructs this function explicitly in terms
of modular forms by using an unexpected integral transform, which establishes
a new connection between modular forms and discrete geometry.

A landmark achievement like Viazovska's deserves to be appreciated by a broad
audience of mathematicians, and indeed it can be.  In this article we'll take
a look at how her proof works, as well as the background and context. We
won't cover all the details completely, but we'll see the main ideas and how
they fit together. Readers who wish to read a complete proof will then be
well prepared to study Viazovska's paper \cite{Viazovska2016} and the
follow-up work on the Leech lattice \cite{CKMRV24}.  See also de Laat and
Vallentin's survey article and interview \cite{dLV2016} for a somewhat
different perspective, as well as \cite{Cohn2016} and \cite{SPLAG} for
further background and references.

\begin{figure}
\begin{center}
\includegraphics[scale=0.86]{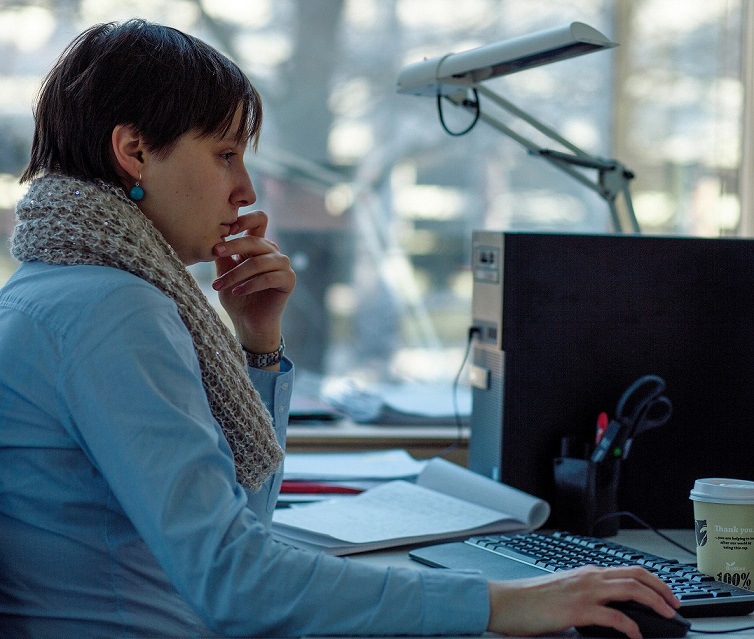}
\end{center}
\caption{Maryna Viazovska solved the sphere packing problem in eight dimensions.}
\label{fig:maryna}
\end{figure}

\begin{figure}
\begin{center}
\includegraphics[scale=0.2]{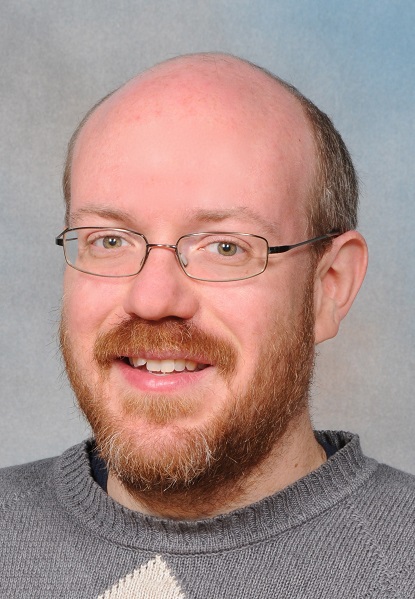}
\includegraphics[scale=0.569825]{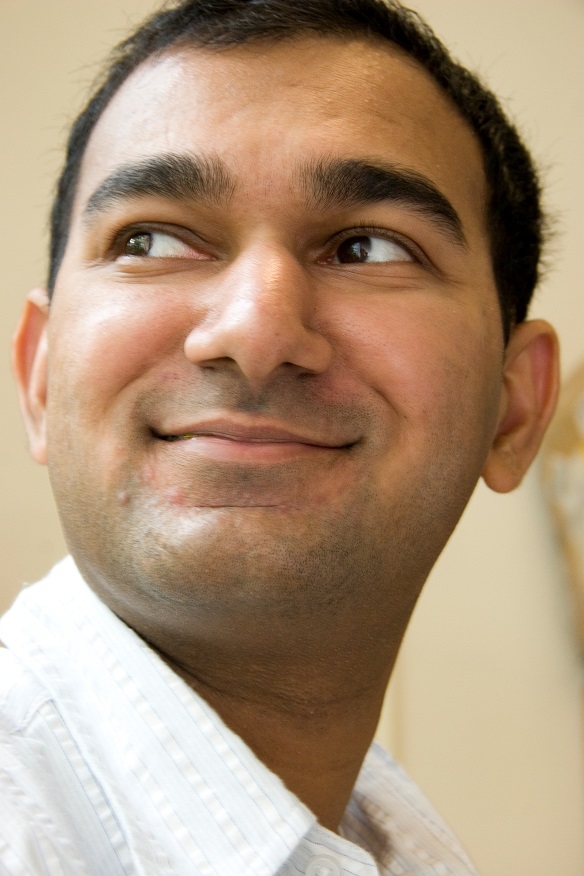}
\includegraphics[scale=0.473236]{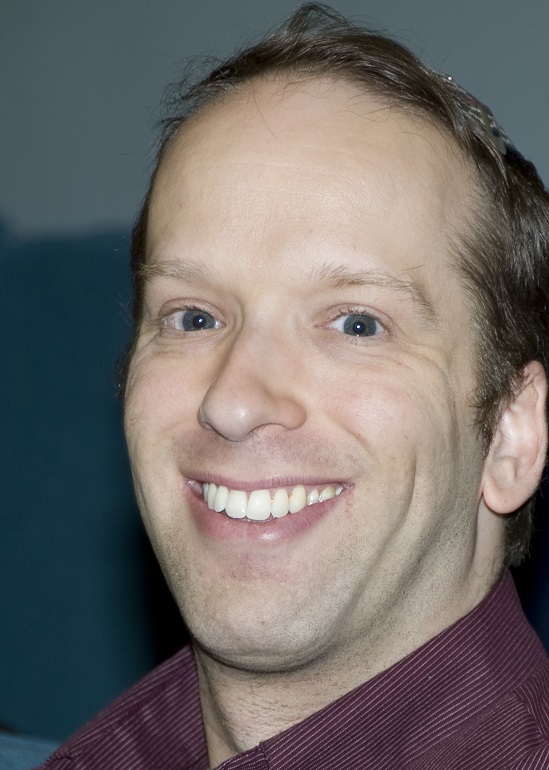}
\includegraphics[scale=0.925852]{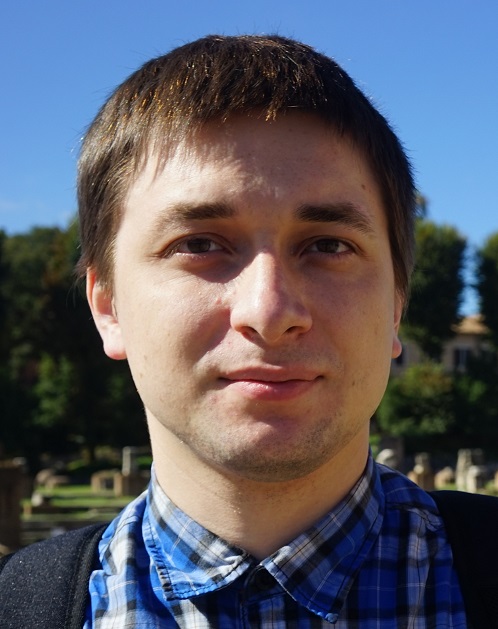}
\end{center}
\caption{Henry Cohn, Abhinav Kumar, Stephen D.\ Miller, and Danylo Radchenko
collaborated with Maryna Viazovska to extend her methods to twenty-four dimensions.}
\label{fig:leech}
\end{figure}

\section{Sphere packing}

The sphere packing problem asks for the densest packing of $\R^n$ with
congruent balls. In other words, what is the largest fraction of $\R^n$ that
can be covered by congruent balls with disjoint interiors?

Pathological packings may not have well-defined densities, but we can handle
the technicalities as follows. A \emph{sphere packing} $\cP$ is a nonempty
subset of $\R^n$ consisting of congruent balls with disjoint interiors. The
\emph{upper density} of $\cP$ is
\[
\limsup_{r \to \infty} \frac{\vol{\big(B_r^n(0) \cap \cP\big)}}{\vol{\big(B_r^n(0)\big)}},
\]
where $B_r^n(x)$ denotes the closed ball of radius $r$ about $x$, and the
\emph{sphere packing density} $\Delta_{\R^n}$ in $\R^n$ is the supremum of
all the upper densities of sphere packings.  In other words, we avoid
technicalities by using a generous definition of the packing density.  This
generosity does not cause any harm, as shown by the theorem of Groemer that
there exists a sphere packing $\cP$ for which
\[
\lim_{r \to \infty} \frac{\vol{\big(B_r^n(x) \cap \cP\big)}}{\vol{\big(B_r^n(x)\big)}} = \Delta_{\R^n}
\]
uniformly for all $x \in \R^n$.  Thus, the supremum of the upper densities is
in fact achieved as the density of some packing, in the nicest possible way.
Of course the densest packing is not unique, since there are any number of
ways to perturb a packing without changing its overall density.

Why should we care about the sphere packing problem?  Two obvious reasons are
that it's a natural geometric problem in its own right and a toy model for
granular materials.  A more surprising application is that sphere packings
are error-correcting codes for a continuous communication channel. Real-world
communication channels can be modeled using high-dimensional vector spaces,
and thus high-dimensional sphere packings have practical importance.

Instead of justifying sphere packing by aspects of the problem or its
applications, we'll justify it by its solutions: a question is good if it has
good answers.  Sphere packing turns out to be a far richer and more beautiful
topic than the bare problem statement suggests.  From this perspective, the
point of the subject is the remarkable structures that arise as dense sphere
packings.

To begin, let's examine the familiar cases of one, two, and three dimensions.
The one-dimensional sphere packing problem is the interval packing problem on
the line, which is of course trivial: the optimal density is $1$.  The two-
and three-dimensional problems are far from trivial, but the optimal
packings, shown in Figure~\ref{fig:low-dimensions}, are exactly what one
would expect.  In particular, the sphere packing density is $\pi/\sqrt{12} =
0.9068\dots$ in $\R^2$ and $\pi/\sqrt{18} = 0.7404\dots$ in $\R^3$. The
two-dimensional problem was solved by Thue.  Giving a rigorous proof requires
a genuine idea, but there exist short, elementary proofs \cite{Hales2000}.
The three-dimensional problem was solved by Hales \cite{Hales2005} via a
lengthy and complex computer-assisted proof, which was extraordinarily
difficult to check but has since been completely verified using formal logic
\cite{FPK}.

In both two and three dimensions, one can obtain an optimal packing by
stacking layers that are packed optimally in the previous dimension, with the
layers nestled together as closely as possible. Guessing this answer is not
difficult, nor is computing the density of such a packing. Instead, the
difficulty lies in proving that no other construction could achieve a greater
density.

\begin{figure}
\begin{center}
\includegraphics[scale=0.163894]{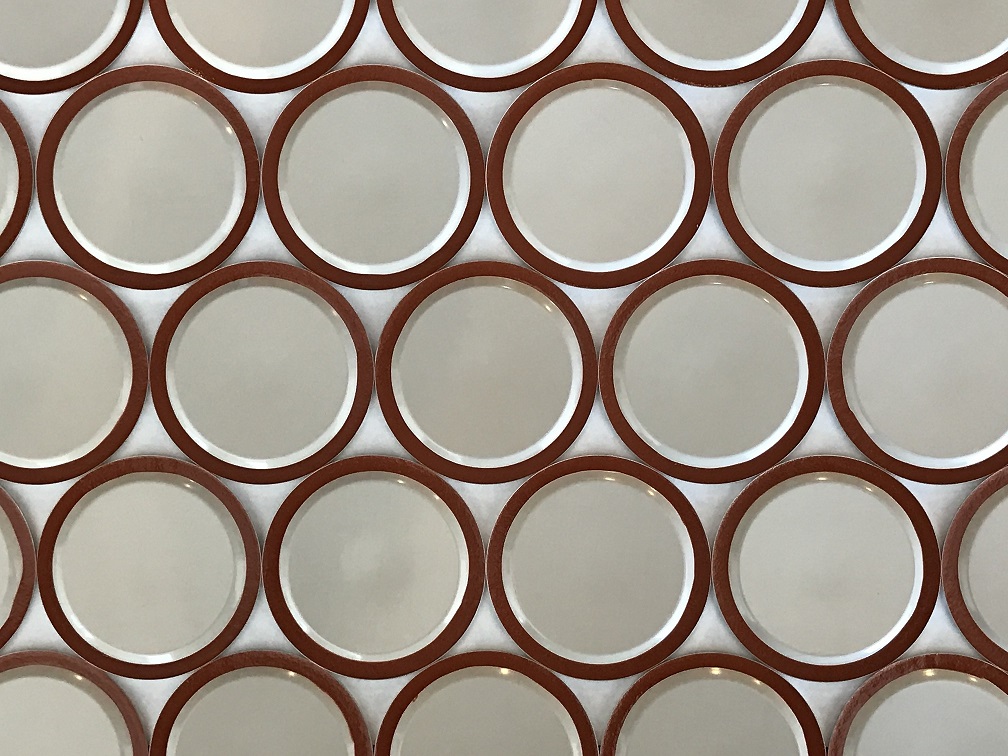}
\includegraphics[scale=0.256]{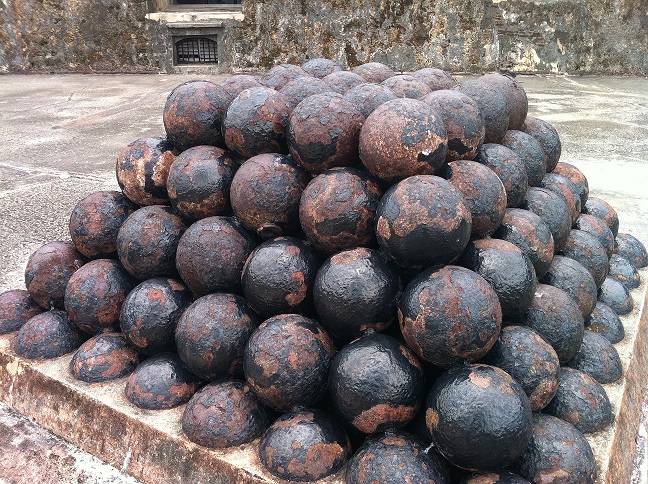}
\end{center}
\caption{Fragments of optimal sphere packings in two and three dimensions,
with density $\pi/\sqrt{12} =
0.9068\dots$ in $\R^2$ and $\pi/\sqrt{18} = 0.7404\dots$ in $\R^3$.}
\label{fig:low-dimensions}
\end{figure}

Unfortunately, our low-dimensional experience is poor preparation for
understanding high-dimensional sphere packing.  Based on the first three
dimensions, it appears that guessing the optimal packing is easy, but this
expectation turns out to be completely false in high dimensions.  In
particular, stacking optimal layers from the previous dimension does not
always yield an optimal packing. (One can recursively determine the best
packings in successive dimensions under such a hypothesis
\cite{ConwaySloane1995}, and this procedure yields a suboptimal packing by
the time it reaches $\R^{10}$.)

The sphere packing problem seems to have no simple, systematic solution that
works across all dimensions. Instead, each dimension has its own
idiosyncracies and charm. Understanding the densest sphere packing in $\R^8$
tells us only a little about $\R^7$ or $\R^9$, and hardly anything about
$\R^{10}$.

Aside from $\R^8$ and $\R^{24}$, our ignorance grows as the dimension
increases. In high dimensions, we have absolutely no idea how the densest
sphere packings behave.  We do not know even the most basic facts, such as
whether the densest packings should be crystalline or disordered. Here ``do
not know'' does not merely mean ``cannot prove,'' but rather the much
stronger ``cannot predict.''

A simple greedy argument shows that the optimal density in $\R^n$ is at least
$2^{-n}$.  To see why, consider any sphere packing in which there is no room
to add even one more sphere.  If we double the radius of each sphere, then
the enlarged spheres must cover space completely, because any uncovered point
could serve as the center of a new sphere that would fit in the original
packing. Doubling the radius multiplies volume by $2^n$, and so the original
packing must cover at least a $2^{-n}$ fraction of $\R^n$.

That may sound appallingly low, but it is very nearly the best lower bound
known. Even the most recent bounds, obtained by Venkatesh
\cite{Venkatesh2013} in 2013, have been unable to improve on $2^{-n}$ by more
than a linear factor in general and an $n \log \log n$ factor in special
cases. As for upper bounds, in 1978 Kabatyanskii and Levenshtein
\cite{KabatyanskiiLevenshtein1978} proved an upper bound of $2^{(-0.599\ldots
+ o(1))n}$, which remains essentially the best upper bound known in high
dimensions. Thus, we know that the sphere packing density decreases
exponentially as a function of dimension, but the best upper and lower bounds
known are exponentially far apart.

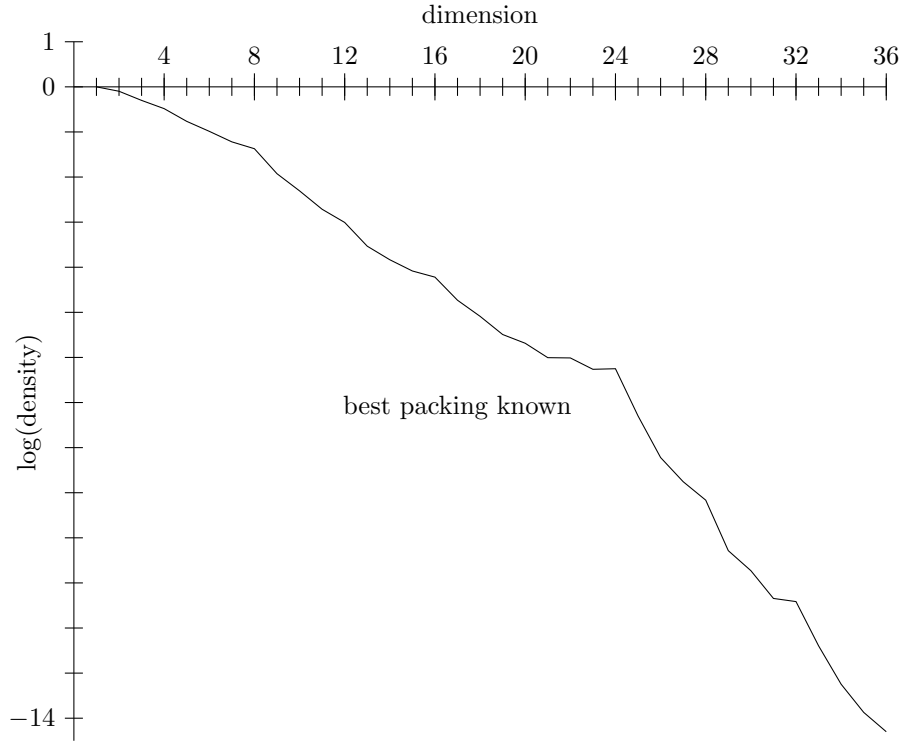
\begin{figure}
\begin{center}
\begin{tikzpicture}[scale=1.2]
\foreach \x in {1,2,3,4,5,6,7,8,9} \draw (\x,-0.15)--(\x,0.15);
\foreach \x in {1,2,3,4,5,6,7,8,9,10,11,12,13,14,15,16,17,18,19,20,21,22,23,24,25,26,27,28,29,30,31,32,33,34,35,36} \draw ({0.25*\x},-0.1)--({0.25*\x},0.1);
\foreach \y in {-14,-13,-12,-11,-10,-9,-8,-7,-6,-5,-4,-3,-2,-1,0,1} \draw (-0.1,{0.5*\y})--(0.1,{0.5*\y});
\draw (4.5,0.6) node[above] {dimension};
\draw (-0.5,-3.5) node[rotate=90] {$\log(\text{density})$};
\draw (4.25,-3.55) node {best packing known};
\draw (1,0.15) node[above] {$4$};
\draw (2,0.15) node[above] {$8$};
\draw (3,0.15) node[above] {$12$};
\draw (4,0.15) node[above] {$16$};
\draw (5,0.15) node[above] {$20$};
\draw (6,0.15) node[above] {$24$};
\draw (7,0.15) node[above] {$28$};
\draw (8,0.15) node[above] {$32$};
\draw (9,0.15) node[above] {$36$};
\draw (-0.1,0.5) node[left] {$1$};
\draw (-0.1,0) node[left] {$0$};
\draw (-0.1,-7) node[left] {$-14$};
\draw (0,-7.25)--(0,0.5);
\draw (0,0)--(9,0);
\draw (0.25,0)--(0.50000,-0.048853)--(0.75000,-0.15022)--(1.0000,-0.24156)--(1.2500,-0.38257)--(1.5000,-0.49315)
--(1.7500,-0.60989)--(2.0000,-0.68586)--(2.2500,-0.96289)--(2.5000,-1.1532)--(2.7500,-1.3572)--(3.0000,-1.5033)
--(3.2500,-1.7666)--(3.5000,-1.9170)--(3.7500,-2.0415)--(4.0000,-2.1097)--(4.2500,-2.3659)--(4.5000,-2.5442)
--(4.7500,-2.7458)--(5.0000,-2.8428)--(5.2500,-3.0026)--(5.5000,-3.0056)--(5.7500,-3.1316)--(6.0000,-3.1252)
--(6.2500,-3.6488)--(6.5000,-4.1100)--(6.7500,-4.3777)--(7.0000,-4.5825)--(7.2500,-5.1425)--(7.5000,-5.3642)
--(7.7500,-5.6721)--(8.0000,-5.7070)--(8.2500,-6.1973)--(8.5000,-6.6231)--(8.7500,-6.9354)--(9.0000,-7.1499);
\end{tikzpicture}
\end{center}
\caption{The sphere packing density is jagged and irregular, with no obvious way to interpolate
    data points from their neighbors.}
\label{fig:densitygraph}
\end{figure}

\begin{table}
\caption{The record sphere packing densities in $\R^n$ with $1 \le n \le 36$, from Table~I.1 of
\cite[pp.~xix--xx]{SPLAG}.  All numbers are rounded down.} \label{table:density}
\begin{center}
\begin{tabular}{cccccccc}
\toprule
$n$ & density & & $n$ & density & & $n$ & density\\
\cmidrule{1-2} \cmidrule{4-5} \cmidrule{7-8}
$1$ & $1.000000000$ & & $13$ & $0.0320142921$ & & $25$ & $0.00067721200977$\\
$2$ & $0.906899682$ & & $14$ & $0.0216240960$ & & $26$ & $0.00026922005043$\\
$3$ & $0.740480489$ & & $15$ & $0.0168575706$ & & $27$ & $0.00015759439072$\\
$4$ & $0.616850275$ & & $16$ & $0.0147081643$ & & $28$ & $0.00010463810492$\\
$5$ & $0.465257613$ & & $17$ & $0.0088113191$ & & $29$ & $0.00003414464690$\\
$6$ & $0.372947545$ & & $18$ & $0.0061678981$ & & $30$ & $0.00002191535344$\\
$7$ & $0.295297873$ & & $19$ & $0.0041208062$ & & $31$ & $0.00001183776518$\\
$8$ & $0.253669507$ & & $20$ & $0.0033945814$ & & $32$ & $0.00001104074930$\\
$9$ & $0.145774875$ & & $21$ & $0.0024658847$ & & $33$ & $0.00000414068828$\\
$10$ & $0.099615782$ & & $22$ & $0.0024510340$ & & $34$ & $0.00000176697388$\\
$11$ & $0.066238027$ & & $23$ & $0.0019053281$ & & $35$ & $0.00000094619041$\\
$12$ & $0.049454176$ & & $24$ & $0.0019295743$ & & $36$ & $0.00000061614660$\\
\bottomrule
\end{tabular}
\end{center}
\end{table}

Table~\ref{table:density} lists the best packing densities currently known in
up to $36$ dimensions, and Figure~\ref{fig:densitygraph} shows a logarithmic
plot.  The plot has several noteworthy features:
\begin{enumerate}
\item The curve is jagged and irregular, with no obvious way to interpolate
    data points from their neighbors.

\item The density is clearly decreasing exponentially, but the irregularity
    makes it unclear how to extrapolate to estimate the decay rate as the
    dimension tends to infinity.

\item There seem to be parity effects.  Even dimensions look slightly
    better than odd dimensions, multiples of four are better yet, and
    multiples of eight are the best of all.

\item Certain dimensions, most notably $24$, have packings so good that
    they seem to pull the entire curve in their direction.  The fact that
    this occurs is not so surprising, since one expects cross sections and
    stackings of great packings to be at least good, but the effect is
    surprisingly large.
\end{enumerate}

\section{Lattices and periodic packings}
\label{sec:lattices}

How can we describe sphere packings?  Random or pathological packings can be
infinitely complicated, but the most important packings can generally be
given a finite description via periodicity.

Recall that a \emph{lattice} in $\R^n$ is a discrete subgroup of rank $n$. In
other words, it consists of the integral span of a basis of $\R^n$.
Equivalently, a lattice is the image of $\Z^n$ under an invertible linear
operator.

A sphere packing $\cP$ is \emph{periodic} if there exists a lattice $\Lambda$
such that $\cP$ is invariant under translation by every element of $\Lambda$.
In that case, the translational symmetry group of $\cP$ must be a lattice,
since it is clearly a discrete group, and $\cP$ consists of finitely many
orbits of this group. A \emph{lattice packing} is a periodic packing in which
the spheres form a single orbit under the translational symmetry group (i.e.,
their centers form a lattice, up to translation).  See
Figure~\ref{fig:periodic} for an illustration.

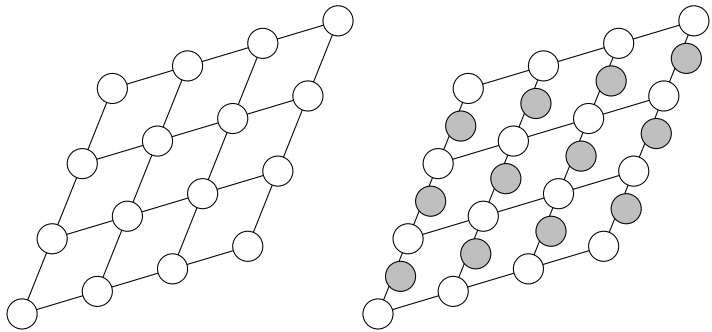
\begin{figure}
\begin{center}
\begin{tikzpicture}
\draw (-1.4,-1.3)--(-0.2,1.7);
\draw (-0.4,-1)--(0.8,2);
\draw (0.6,-0.7)--(1.8,2.3);
\draw (1.6,-0.4)--(2.8,2.6);
\draw (-1.4,-1.3)--(1.6,-0.4);
\draw (-1,-0.3)--(2,0.6);
\draw (-0.6,0.7)--(2.4,1.6);
\draw (-0.2,1.7)--(2.8,2.6);
\foreach \x in {-1,0,1,2}
\foreach \y in {-1,0,1,2}
\fill[white] ({0.4*\x+\y},{\x+0.3*\y}) circle (0.2);
\foreach \x in {-1,0,1,2}
\foreach \y in {-1,0,1,2}
\draw ({0.4*\x+\y},{\x+0.3*\y}) circle (0.2);
\end{tikzpicture}
\begin{tikzpicture}
\draw (-1.4,-1.3)--(-0.2,1.7);
\draw (-0.4,-1)--(0.8,2);
\draw (0.6,-0.7)--(1.8,2.3);
\draw (1.6,-0.4)--(2.8,2.6);
\draw (-1.4,-1.3)--(1.6,-0.4);
\draw (-1,-0.3)--(2,0.6);
\draw (-0.6,0.7)--(2.4,1.6);
\draw (-0.2,1.7)--(2.8,2.6);
\foreach \x in {-1,0,1}
\foreach \y in {-1,0,1,2}
\fill[black!25] ({0.4*\x+\y+0.3},{\x+0.3*\y+0.5}) circle (0.2);
\foreach \x in {-1,0,1,2}
\foreach \y in {-1,0,1,2}
\fill[white] ({0.4*\x+\y},{\x+0.3*\y}) circle (0.2);
\foreach \x in {-1,0,1,2}
\foreach \y in {-1,0,1,2}
\draw ({0.4*\x+\y},{\x+0.3*\y}) circle (0.2);
\foreach \x in {-1,0,1}
\foreach \y in {-1,0,1,2}
\draw ({0.4*\x+\y+0.3},{\x+0.3*\y+0.5}) circle (0.2);
\end{tikzpicture}
\end{center}
\caption{The spheres in a lattice packing form a single orbit under
translation (left), while those in a periodic packing can form several orbits
(right).  The small parallelograms are fundamental cells.}
\label{fig:periodic}
\end{figure}

It is not known whether periodic packings attain the optimal sphere packing
density in each dimension, aside from the five cases in which the sphere
packing problem has been solved. They certainly come arbitrarily close to the
optimal density: given an optimal packing, one can approximate it by taking
the spheres contained in a large box and repeating them periodically
throughout space, and the density loss is negligible if the box is large
enough. However, there seems to be no reason why periodic packings should
reach the exact optimum, and perhaps they don't in high dimensions.

By contrast, lattices probably do not even come arbitrarily close to the
optimal packing density in high dimensions. For example, the best periodic
packing known in $\R^{10}$ is more than 8\% denser than the best lattice
packing known.  Seen in this light, the optimality of lattices in $\R^8$ and
$\R^{24}$ is not a foregone conclusion, but rather an indication that sphere
packing in these dimensions is particularly simple.

To compute the density of a lattice packing, it's convenient to view the
lattice as a tiling of space with parallelotopes (the $n$-dimensional
analogue of parallelograms).  Given a basis $v_1,\dots,v_n$ for a lattice
$\Lambda$, the parallelotope
\[
\{ x_1 v_1 + \cdots + x_n v_n : \text{$0 \le x_i < 1$ for $i=1,2,\dots,n$}\}
\]
is called the \emph{fundamental cell} of $\Lambda$ with respect to this
basis. Translating the fundamental cell by elements of $\Lambda$ tiles
$\R^n$, as in Figure~\ref{fig:periodic}.  From this perspective, a lattice
sphere packing amounts to placing spheres at the vertices of such a tiling.
On a global scale, there is one sphere for each copy of the fundamental cell.
Thus, if the packing uses spheres of radius $r$ and has fundamental cell $C$,
then its density is the ratio
\[
\frac{\vol{\big(B_r^n\big)}}{\vol{(C)}}.
\]

Both factors in this ratio are easily computed if we are given $r$ and $C$.
The volume of a fundamental cell is just the absolute value of the
determinant of the corresponding lattice basis; we will write it as
$\vol{(\R^n/\Lambda)}$, the volume of the quotient torus, to avoid having to
specify a basis. Computing the volume of a ball of radius $r$ in $\R^n$ is a
multivariate calculus exercise, whose answer is
\[
\vol{\big(B_r^n\big)} = \frac{\pi^{n/2}}{(n/2)!}r^n,
\]
where of course $(n/2)!$ means $\Gamma(n/2+1)$ when $n$ is odd.  We can
therefore compute the density of any lattice packing explicitly. The density
of a periodic packing is equally easy to compute: if the packing consists of
$N$ translates of a lattice $\Lambda$ in $\R^n$ and uses spheres of radius
$r$, then its density is
\[
\frac{N \vol{\big(B_r^n\big)}}{\vol{(\R^n/\Lambda)}}.
\]

Of course the density of a packing depends on the radius of the spheres.
Given a lattice with no radius specified, it is standard to use the largest
radius that does not lead to overlap.  The \emph{minimal vector length} of a
lattice $\Lambda$ is the length of the shortest nonzero vector in $\Lambda$,
or equivalently the shortest distance between two distinct points in
$\Lambda$. If the minimal vector length is $r$, then $r/2$ is the largest
radius that yields a packing, since that is the radius at which neighboring
spheres become tangent.

\section{The $E_8$ and Leech lattices}

Many dimensions feature noteworthy sphere packings, but the $E_8$ root
lattice in $\R^8$ and the Leech lattice in $\R^{24}$ are perhaps the most
remarkable of all, with connections to exceptional structures across
mathematics. In this section, we'll construct $E_8$ and prove some of its
basic properties. It was discovered by Korkine and Zolotareff in 1873, in the
guise of a quadratic form they called $W_8$. We'll give a construction much
like Korkine and Zolotareff's but more modern. The Leech lattice
$\Lambda_{24}$, discovered by Leech in 1967, is similar in spirit, but more
complicated. In lieu of constructing it, we will briefly summarize its
properties.

To specify $E_8$, we just need to describe a lattice basis $v_1,\dots,v_8$ in
$\R^8$. Furthermore, only the relative positions of the basis vectors matter,
so all we need to specify is their inner products with each other.  All this
information will be encoded by the \emph{Dynkin diagram}
\begin{center}
\begin{tikzpicture}
\draw (-2.25,0)--(2.25,0);
\draw (-0.75,0)--(-0.75,0.75);
\fill[white] (-2.25,0) circle (0.16);
\fill[white] (-1.5,0) circle (0.16);
\fill[white] (-0.75,0) circle (0.16);
\fill[white] (-0.75,0.75) circle (0.16);
\fill[white] (0,0) circle (0.16);
\fill[white] (0.75,0) circle (0.16);
\fill[white] (1.5,0) circle (0.16);
\fill[white] (2.25,0) circle (0.16);
\draw (-2.25,0) circle (0.16);
\draw (-1.5,0) circle (0.16);
\draw (-0.75,0) circle (0.16);
\draw (-0.75,0.75) circle (0.16);
\draw (0,0) circle (0.16);
\draw (0.75,0) circle (0.16);
\draw (1.5,0) circle (0.16);
\draw (2.25,0) circle (0.16);
\end{tikzpicture}
\end{center}
of $E_8$.  In this diagram, the eight nodes correspond to the basis vectors,
each of squared length $2$.  The inner product between distinct vectors is
$-1$ if the nodes are joined by an edge, and $0$ otherwise.  Thus, if we
number the nodes
\begin{center}
\begin{tikzpicture}
\draw (-2.25,0)--(2.25,0);
\draw (-0.75,0)--(-0.75,0.75);
\fill[white] (-2.25,0) circle (0.16);
\fill[white] (-1.5,0) circle (0.16);
\fill[white] (-0.75,0) circle (0.16);
\fill[white] (-0.75,0.75) circle (0.16);
\fill[white] (0,0) circle (0.16);
\fill[white] (0.75,0) circle (0.16);
\fill[white] (1.5,0) circle (0.16);
\fill[white] (2.25,0) circle (0.16);
\draw (-2.25,0) circle (0.16);
\draw (-1.5,0) circle (0.16);
\draw (-0.75,0) circle (0.16);
\draw (-0.75,0.75) circle (0.16);
\draw (0,0) circle (0.16);
\draw (0.75,0) circle (0.16);
\draw (1.5,0) circle (0.16);
\draw (2.25,0) circle (0.16);
\draw (-2.25,0) node {{\footnotesize 1}};
\draw (-1.5,0) node {{\footnotesize 2}};
\draw (-0.75,0) node {{\footnotesize 3}};
\draw (-0.75,0.75) node {{\footnotesize 4}};
\draw (0,0) node {{\footnotesize 5}};
\draw (0.75,0) node {{\footnotesize 6}};
\draw (1.5,0) node {{\footnotesize 7}};
\draw (2.25,0) node {{\footnotesize 8}};
\end{tikzpicture}
\end{center}
then the Gram matrix of inner products for this basis is given by
\begin{equation} \label{eq:Gram}
\big(\langle v_i,v_j \rangle\big)_{1\le i,j \le 8} =
\begin{bmatrix}
2 & -1 & 0 & 0 & 0 & 0 & 0 & 0\\
-1 & 2 & -1 & 0 & 0 & 0 & 0 & 0\\
0 & -1 & 2 & -1 & -1 & 0 & 0 & 0\\
0 & 0 & -1 & 2 & 0 & 0 & 0 & 0\\
0 & 0 & -1 & 0 & 2 & -1 & 0 & 0\\
0 & 0 & 0 & 0 & -1 & 2 & -1 & 0\\
0 & 0 & 0 & 0 & 0 & -1 & 2 & -1\\
0 & 0 & 0 & 0 & 0 & 0 & -1 & 2
\end{bmatrix}.
\end{equation}

Before we go further, we must address a fundamental question: how do we know
there really are vectors $v_1,\dots,v_8$ with these inner products?  All we
need is for the matrix in \eqref{eq:Gram} to be symmetric and positive
definite, and indeed it is, although it's not obviously positive definite.
That can be checked in several ways. We'll take the pedestrian approach of
observing that the characteristic polynomial of this matrix is
\[
t^8 - 16t^7 + 105t^6 - 364t^5 + 714t^4 - 784t^3 + 440t^2 - 96t + 1,
\]
which clearly has no roots when $t<0$ because every term is then positive.

We can now define the \emph{$E_8$ root lattice} to be the integral span of
$v_1,\dots,v_8$.  We will use this definition to derive several fundamental
properties of $E_8$.  These properties will let us determine its packing
density, and they will also be essential for Viazovska's proof.

The $E_8$ lattice is an \emph{integral lattice}, which means all the inner
products between vectors in $E_8$ are integers.  This follows immediately
from the integrality of the inner products of the basis vectors
$v_1,\dots,v_8$. Even more importantly, $E_8$ is an \emph{even lattice},
which means the squared length of every vector is an even integer.
Specifically, for $m_1,\dots,m_8 \in \Z$ the vector $m_1v_1+\dots+m_8v_8$ has
squared length
\[
|m_1 v_1 + \dots + m_8 v_8|^2 = 2m_1^2 + \dots + 2m_8^2 +
\sum_{1 \le i < j \le 8} 2 m_i m_j \langle v_i,v_j \rangle,
\]
which is visibly even.  Thus, the distances between distinct points in $E_8$
are all of the form $\sqrt{2k}$ with $k=1,2,\dots$, and in fact each of those
distances does occur.

In particular, the distance between neighboring points in $E_8$ is
$\sqrt{2}$, so we can form a packing with spheres of radius $\sqrt{2}/2$ and
density
\[
\frac{\vol{\left(B^8_{\sqrt{2}/2}\right)}}{\vol{(\R^8/E_8)}} = \frac{\pi^4}{384 \vol{(\R^8/E_8)}}.
\]
To compute the density of the $E_8$ packing, all we need to compute is
$\vol{(\R^8/E_8)}$.

To compute this volume, recall that it's the absolute value of the
determinant of the basis matrix:
\[
\vol{(\R^8/E_8)} = \left|\det\begin{bmatrix}
\longleftarrow v_1 \longrightarrow\\
\longleftarrow v_2 \longrightarrow\\
\vdots\\
\longleftarrow v_8 \longrightarrow
\end{bmatrix}\right|.
\]
However, we can write the Gram matrix $\big(\langle v_i,v_j
\rangle\big)_{1\le i,j \le 8}$ as the product
\[
\begin{bmatrix}
\longleftarrow v_1 \longrightarrow\\
\longleftarrow v_2 \longrightarrow\\
\vdots\\
\longleftarrow v_8 \longrightarrow
\end{bmatrix}
\begin{bmatrix}
{\Big\uparrow} & {\Big\uparrow} & & {\Big\uparrow}\\
v_1 & v_2 & \cdots & v_8\\
{\Big\downarrow} & {\Big\downarrow} & & {\Big\downarrow}\\
\end{bmatrix}
\]
of the basis matrix with its transpose, and thus
\[
\det \big(\langle v_i,v_j \rangle\big)_{1\le i,j \le 8} =
\vol{(\R^8/E_8)}^2.
\]
Computing the determinant of the matrix in \eqref{eq:Gram} then shows that
$\vol{(\R^8/E_8)} = 1$. In other words, $E_8$ is a \emph{unimodular} lattice.

Putting together our calculations, we have proved the following proposition:

\begin{proposition}
The $E_8$ lattice packing in $\R^8$ has density $\pi^4/384 = 0.2536\dots$.
\end{proposition}

Our calculations so far have led us to what turns out to be the densest
sphere packing in $\R^8$, but it's not obvious from this construction that
$E_8$ is an especially interesting lattice. The $E_8$ lattice is in fact
magnificently symmetrical, far more so than one might naively guess based on
its lopsided Dynkin diagram. Its symmetry group is the $E_8$ Weyl group,
which is generated by reflections in the hyperplanes orthogonal to each of
$v_1,\dots,v_8$.  We will not make use of this group, but it's important to
keep in mind that the lattice itself is far more symmetrical than its
definition.  This is a common pattern when defining highly symmetrical
objects.

Our density calculation for $E_8$ was based on its being an even unimodular
lattice.  In fact, $E_8$ is the unique even unimodular lattice in $\R^8$, up
to orthogonal transformations.  Even unimodular lattices exist only when the
dimension is a multiple of eight, and they play a surprisingly large role in
the theory of sphere packing.

\begin{figure}
\begin{center}
\includegraphics[scale=0.28]{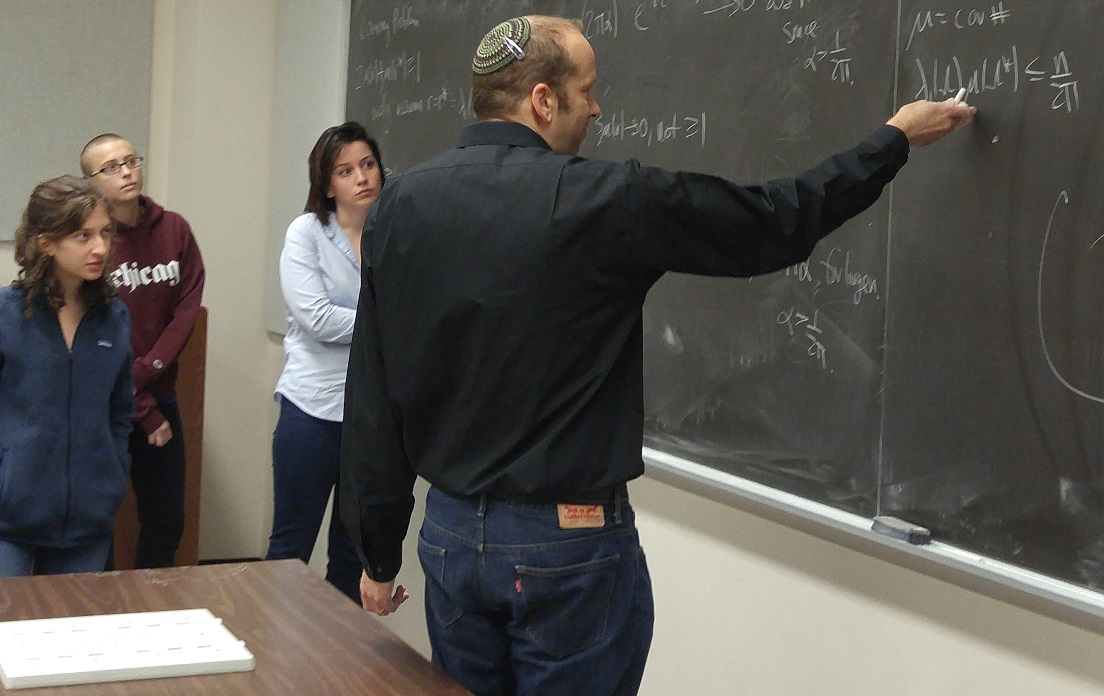}
\end{center}
\caption{Stephen D.\ Miller explains dual lattices and transference theorems
to his graduate class on the geometry of numbers.}
\label{fig:steve-class}
\end{figure}

The last property of $E_8$ we will need for Viazovska's proof is that it is
its own dual lattice, a concept we will define shortly.  Given a lattice
$\Lambda$ with basis $v_1,\dots,v_n$, let $v_1^*,\dots,v_n^*$ be the dual
basis with respect to the usual inner product. In other words,
\[
\langle v_i, v_j^* \rangle = \begin{cases} 1 & \text{if $i=j$, and}\\
0 & \text{otherwise.}
\end{cases}
\]
Then the \emph{dual lattice} $\Lambda^*$ of $\Lambda$ is the lattice with
basis $v_1^*,\dots,v_n^*$.  It is not difficult to check that $\Lambda^*$ is
independent of the choice of basis for $\Lambda$; one basis-free way to
characterize it is that
\begin{equation} \label{eq:dual-lattice}
\Lambda^* = \{y \in \R^n : \text{$\langle x,y \rangle \in \Z$ for all $x \in \Lambda$}\}.
\end{equation}
The self-duality $E_8^*=E_8$ is a consequence of the following lemma:

\begin{lemma}
Every integral unimodular lattice $\Lambda$ satisfies $\Lambda^*=\Lambda$.
\end{lemma}

\begin{proof}
Let $v_1,\dots,v_n$ be a basis of $\Lambda$, and $v_1^*,\dots,v_n^*$ the dual
basis of $\Lambda^*$. By construction, the basis matrix formed by
$v_1^*,\dots,v_n^*$ is the inverse of the transpose of that formed by
$v_1,\dots,v_n$, and hence
$\vol{(\R^n/\Lambda^*)} = 1/\!\vol{(\R^n/\Lambda)}$. % ad hoc \!
If $\Lambda$ is an integral lattice, then $\Lambda \subseteq \Lambda^*$, and
the index of $\Lambda$ in $\Lambda^*$ is given by
\[
[\Lambda^* : \Lambda] = \vol(\R^n/\Lambda)/\!\vol{(\R^n/\Lambda^*)} =
\vol{(\R^n/\Lambda)}^2.
\]
If $\Lambda$ is unimodular as well, then $[\Lambda^* : \Lambda]=1$ and hence
$\Lambda^* = \Lambda$.
\end{proof}

As mentioned above, the Leech lattice $\Lambda_{24}$ is similar to $E_8$ but
more elaborate.  It's an even unimodular lattice in $\R^{24}$, but this time
with no vectors of length $\sqrt{2}$, and it's the unique lattice with these
properties, up to orthogonal transformations. The nonzero vectors in
$\Lambda_{24}$ have lengths $\sqrt{2k}$ for $k=2,3,\dots$, and of course
$\Lambda_{24}^* = \Lambda_{24}$ because $\Lambda_{24}$ is integral and
unimodular. One noteworthy property of $\Lambda_{24}$ is that it's chiral:
all of its symmetries are orientation-preserving, and the Leech lattice
therefore occurs in left-handed and right-handed variants, which are mirror
images of each other. (By contrast, the symmetry group of $E_8$ is generated
by reflections, so $E_8$ is certainly not chiral.)

The sphere packing density of the Leech lattice is
\[
\frac{\vol{\left(B^{24}_{1}\right)}}{\vol{(\R^{24}/\Lambda_{24})}} = \frac{\pi^{12}}{12!} = 0.001929\dots,
\]
which looks awfully low, but keep in mind that the optimal density decreases
exponentially as a function of dimension.  In fact, the density of the Leech
lattice is remarkably high, as one can see from Figure~\ref{fig:densitygraph}
and Table~\ref{table:density}.  For comparison, the best density known in
$\R^{23}$ is $0.001905\dots$, which is lower than the density of the Leech
lattice, and this is the only case in which the density increases from one
dimension to the next in Table~\ref{table:density}.

\section{Linear programming bounds}

The underlying technique used in Viazovska's proof is \emph{linear
programming bounds} for the sphere packing density in $\R^n$.  These upper
bounds were developed by Cohn and Elkies \cite{CohnElkies2003}, based on
several decades of previous work initiated by Delsarte and extended by
numerous mathematicians.  In this approach to sphere packing, one uses
auxiliary functions with certain properties to obtain density bounds.
Viazovska's breakthrough consists of a new technique for constructing these
auxiliary functions, but before we turn to her proof let's examine the
general theory and review how the bounds work.  We will see that the general
bounds do not refer to special dimensions such as eight and twenty-four,
which makes it all the more remarkable that they can be used to solve the
sphere packing problem in these dimensions.

\begin{figure}
\begin{center}
\includegraphics[scale=0.25]{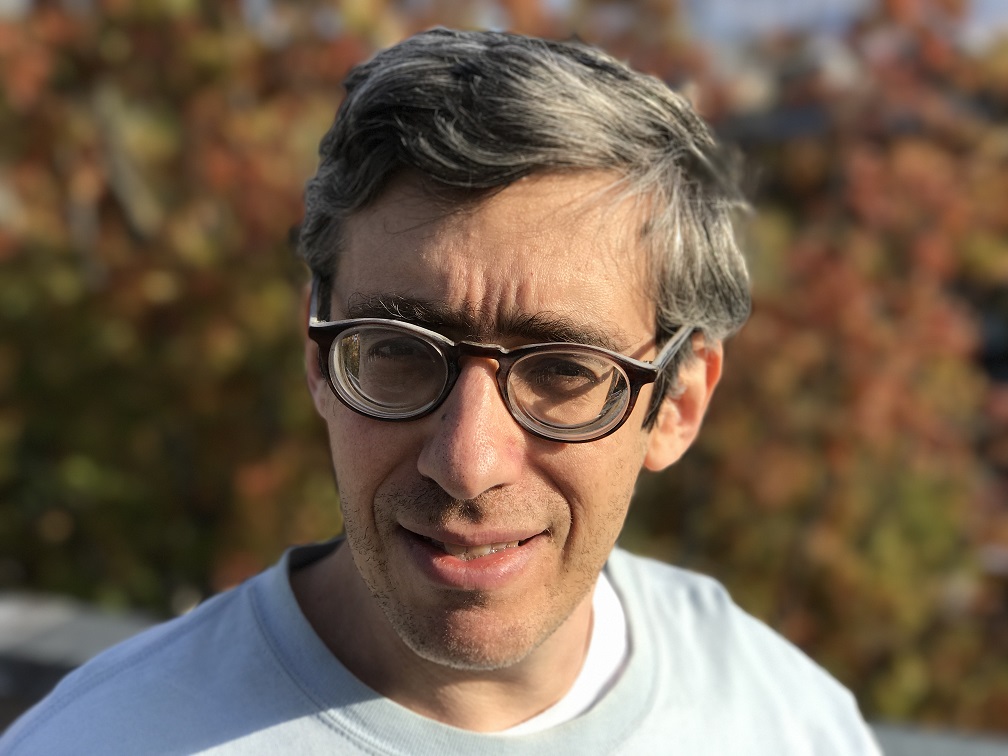}
\end{center}
\caption{Noam Elkies developed the linear programming bounds for sphere
packing with Henry Cohn.}
\label{fig:noam}
\end{figure}

Linear programming bounds are based on harmonic analysis.  That may sound
surprising, since sphere packing is a problem in discrete geometry, which at
first glance seems to have little to do with the continuous problems studied
in harmonic analysis.  However, there is a deep connection between these
fields, because the Fourier transform is essential for understanding the
action of the additive group $\R^n$ on itself by translation, so much so that
one can't truly understand lattices without harmonic analysis.

Define the \emph{Fourier transform} $\widehat{f}$ of an integrable function
$f \colon \R^n \to \R$ by
\[
\widehat{f}(y) = \int_{\R^n} f(x) e^{-2\pi i \langle x,y \rangle} \, dx.
\]
\emph{Fourier inversion} tells us that if $\widehat{f}$ is integrable as
well, then one can similarly recover $f$ from $\widehat{f}$:
\begin{equation} \label{eq:fourier-inversion}
f(x) = \int_{\R^n} \widehat{f}(y) e^{2\pi i \langle x,y \rangle} \, dy
\end{equation}
almost everywhere.  In other words, the Fourier transform gives the unique
coefficients needed to express $f$ in terms of complex exponentials.

To avoid analytic technicalities, we will focus on Schwartz functions. Recall
that $f \colon \R^n \to \R$ is a \emph{Schwartz function} if $f$ is
infinitely differentiable, $f(x) = O\big((1+|x|)^{-k}\big)$ for all
$k=1,2,\dots$, and the same holds for all the partial derivatives of $f$ (of
every order). Schwartz functions behave particularly well, well enough to
justify everything we'd like to do with them, and they are closed under the
Fourier transform. We could get by with weaker hypotheses, but in fact
Viazovska's construction produces Schwartz functions, so we might as well
focus on that case.

The significance of the Fourier transform in sphere packing is that it
diagonalizes the operation of translation by any vector.  Specifically,
\eqref{eq:fourier-inversion} implies that
\[
f(x+t) = \int_{\R^n} \widehat{f}(y) e^{2\pi i \langle t,y \rangle} e^{2\pi i \langle x,y \rangle}\, dy,
\]
which means that translating the input to the function $f$ by $t$ amounts to
multiplying its Fourier transform $\widehat{f}(y)$ by $e^{2\pi i \langle t,y
\rangle}$.  Simultaneously diagonalizing all these translation operators
makes the Fourier transform an ideal tool for studying periodic structures.

The key technical tool behind linear programming bounds is the \emph{Poisson
summation formula}, which expresses a duality between summing a function over
a lattice and summing the Fourier transform over the dual lattice, as defined
in \eqref{eq:dual-lattice}.  Poisson summation says that if $f$ is a Schwartz
function, then
\begin{equation} \label{eq:poisson}
\sum_{x \in \Lambda} f(x) = \frac{1}{\vol{(\R^n/\Lambda)}} \sum_{y \in \Lambda^*} \widehat{f}(y).
\end{equation}
In other words, summing $\widehat{f}$ over $\Lambda^*$ is almost the same as
summing $f$ over $\Lambda$, with the only difference being a factor of
$\vol{(\R^n/\Lambda)}$.  When expressed in this form, Poisson summation looks
mysterious, but it becomes far more transparent when written in the
translated form
\begin{equation} \label{eq:poisson-general}
\sum_{x \in \Lambda} f(x+t) = \frac{1}{\vol{(\R^n/\Lambda)}} \sum_{y \in \Lambda^*} \widehat{f}(y) e^{2\pi i \langle y,t \rangle}.
\end{equation}
This equation reduces to \eqref{eq:poisson} when $t=0$, and it has a simple
proof.  As a function of $t$, the left side of \eqref{eq:poisson-general} is
periodic modulo $\Lambda$, while the right side is its Fourier series.  In
particular, the right side uses exactly the complex exponentials $t \mapsto
e^{2\pi i \langle y,t \rangle}$ that are periodic modulo $\Lambda$, namely
those with $y \in \Lambda^*$ (as follows easily from
\eqref{eq:dual-lattice}). Orthogonality let us compute the coefficient of
such an exponential, and some manipulation yields
$\widehat{f}(y)/\!\vol{(\R^n/\Lambda)}$.  % ad hoc \!

Now we can state and prove the linear programming bounds, which show how to
convert a certain sort of auxiliary function into a sphere packing bound.
Specifically, we will use functions $f \colon \R^n \to \R$ such that $f$ is
eventually nonpositive (i.e., there exists a radius $r$ such that $f(x) \le
r$ for $|x| \ge r$) while $\widehat{f}$ is nonnegative everywhere.

\begin{theorem}[Cohn and Elkies \cite{CohnElkies2003}] \label{thm:LPbounds}
Let $f \colon \R^n \to \R$ be a Schwartz function and $r$ a positive real
number such that $f(0) = \widehat{f}(0) > 0$, $\widehat{f}(y) \ge 0$ for all
$y \in \R^n$, and $f(x) \le 0$ for $|x| \ge r$.  Then the sphere packing
density in $\R^n$ is at most $\vol{\big(B_{r/2}^n\big)}$.
\end{theorem}

The name ``linear programming'' refers to optimizing a linear function
subject to linear constraints.  The optimization problem of choosing $f$ so
as to minimize $r$ can be rephrased as an infinite-dimensional linear program
after a change of variables, but we will not adopt that perspective here.

\begin{proof}
The proof consists of applying the contrasting inequalities $f(x) \le 0$ and
$\widehat{f}(y) \ge 0$ to the two sides of Poisson summation. We will begin
by proving the theorem for lattice packings, which is the simplest case.

Suppose $\Lambda$ is a lattice in $\R^n$, and suppose without loss of
generality that the minimal vector length of $\Lambda$ is $r$ (since the
sphere packing density is invariant under rescaling). In other words, the
packing uses balls of radius $r/2$, and its density is
\[
\frac{\vol{\big(B_{r/2}^n\big)}}{\vol{(\R^n/\Lambda)}}.
\]
Proving the desired density bound $\vol{\big(B_{r/2}^n\big)}$ for $\Lambda$
amounts to showing that $\vol{(\R^n/\Lambda)} \ge 1$. By Poisson summation,
\begin{equation} \label{eq:apply-poisson}
\sum_{x \in \Lambda} f(x) = \frac{1}{\vol{(\R^n/\Lambda)}} \sum_{y \in \Lambda^*} \widehat{f}(y).
\end{equation}
Now the inequality $f(x) \le 0$ for $|x|\ge r$ tells us that the left side of
\eqref{eq:apply-poisson} is bounded above by $f(0)$, and the inequality
$\widehat{f}(y) \ge 0$ tells us that the right side is bounded below by
$\widehat{f}(0)/\!\vol{(\R^n/\Lambda)}$.  % ad hoc \!
It follows that
\[
f(0) \ge \frac{\widehat{f}(0)}{\vol{(\R^n/\Lambda)}},
\]
which yields $\vol{(\R^n/\Lambda)} \ge 1$ because $f(0) = \widehat{f}(0)
> 0$.

The general case is almost as simple, but the algebraic manipulations are a
little trickier.  Because periodic packings come arbitrarily close to the
optimal sphere packing density, without loss of generality we can consider a
periodic packing using balls of radius $r/2$, centered at the translates of a
lattice $\Lambda \subseteq \R^n$ by vectors $t_1,\dots,t_N$.  The packing
density is
\[
\frac{N\vol{\big(B_{r/2}^n\big)}}{\vol{(\R^n/\Lambda)}},
\]
and so we wish to prove that $\vol{(\R^n/\Lambda)} \ge N$.

We will use the translated Poisson summation formula
\eqref{eq:poisson-general}, which after a little manipulation implies that
\[
\sum_{j,k=1}^N \sum_{x \in \Lambda} f(t_j-t_k+x) =
\frac{1}{\vol{(\R^n/\Lambda)}} \sum_{y \in \Lambda^*} \widehat{f}(y) \left|\sum_{j=1}^N e^{2\pi i \langle y,t_j \rangle}\right|^2.
\]
Again we apply the contrasting inequalities on $f$ and $\widehat{f}$ to the
left and right sides, respectively.  On the left, we obtain an upper bound by
throwing away every term except when $j=k$ and $x=0$; on the right, we obtain
a lower bound by throwing away every term except when $y=0$.  Thus,
\[
N f(0) \ge \frac{N^2}{\vol{(\R^n/\Lambda)}} \widehat{f}(0),
\]
which implies that $\vol{(\R^n/\Lambda)} \ge N$ and hence that the density is
at most $\vol{\big(B_{r/2}^n\big)}$, as desired.
\end{proof}

This proof technique may look absurdly inefficient.  We start with Poisson
summation, which expresses a deep duality, and then we recklessly throw away
all the nontrivial terms, leaving only the contributions from the origin. One
practical justification is that we have little choice in the matter, since we
don't know what the other terms are (they all depend on the lattice). A
deeper justification is that the omitted terms are generally small, and
sometimes zero, so omitting them is not as bad as it sounds.

To apply Theorem~\ref{thm:LPbounds}, we must choose an auxiliary function
$f$.  The theorem then shows how to obtain a density bound from $f$, but it
says nothing about how to choose $f$ so as to minimize $r$ and hence minimize
the density bound. Sadly, optimizing the auxiliary function remains an
unsolved problem, and the best possible choice of $f$ is known only when
$n=1$, $8$, or $24$.

As a first step towards solving this problem, note that we can radially
symmetrize $f$, so that $f(x)$ depends only on $|x|$, because all the
constraints on $f$ are linear and rotationally invariant.  Then $f$ is really
a function of one radial variable, as is $\widehat{f}$.  Functions of one
variable feel like they should be tractable, but this optimization problem
turns out to be impressively subtle.

If we can't fully optimize the choice of $f$, then what can we do?  Several
explicit constructions are known, but in general we must resort to numerical
computation.  For this purpose, it's convenient to use auxiliary functions of
the form $f(x) = p(|x|^2) e^{-\pi |x|^2}$, where $p$ is a polynomial.  These
functions are flexible enough to approximate arbitrary radial Schwartz
functions, but simple enough to be tractable.  Numerical optimization then
yields a high-precision approximation to the linear programming bound, which
is shown in Figure~\ref{fig:LPgraph} and Table~\ref{table:lpbound}.

\section{The hunt for the magic functions}

The most striking property of Figure~\ref{fig:LPgraph} is that the upper and
lower bounds in $\R^n$ seem to touch when $n=8$ or $24$.  In other words,
there should be \emph{magic auxiliary functions} that solve the sphere
packing problem in these dimensions, by achieving $r=\sqrt{2}$ in
Theorem~\ref{thm:LPbounds} when $n=8$ and $r=2$ when $n=24$.  (These values
of $r$ are the minimal vector lengths in $E_8$ and $\Lambda_{24}$,
respectively.)  This is exactly what has now been proved, and the proof
simply amounts to constructing an appropriate auxiliary function. Linear
programming bounds do not seem to be sharp for any other $n>2$, which makes
these two cases truly remarkable.

% Figure and table would make more sense in previous section, but putting them
% here gets them a page by themselves.
\begin{figure}
\begin{center}
\begin{tikzpicture}[scale=1.2]
\foreach \x in {1,2,3,4,5,6,7,8,9} \draw (\x,-0.15)--(\x,0.15);
\foreach \x in {1,2,3,4,5,6,7,8,9,10,11,12,13,14,15,16,17,18,19,20,21,22,23,24,25,26,27,28,29,30,31,32,33,34,35,36} \draw ({0.25*\x},-0.1)--({0.25*\x},0.1);
\foreach \y in {-14,-13,-12,-11,-10,-9,-8,-7,-6,-5,-4,-3,-2,-1,0,1} \draw (-0.1,{0.5*\y})--(0.1,{0.5*\y});
\draw (4.5,0.6) node[above] {dimension};
\draw (-0.5,-3.5) node[rotate=90] {$\log(\text{density})$};
\draw (5.15,-0.9) node {linear programming bound};
\draw (4.25,-3.55) node {best packing known};
\draw (1,0.15) node[above] {$4$};
\draw (2,0.15) node[above] {$8$};
\draw (3,0.15) node[above] {$12$};
\draw (4,0.15) node[above] {$16$};
\draw (5,0.15) node[above] {$20$};
\draw (6,0.15) node[above] {$24$};
\draw (7,0.15) node[above] {$28$};
\draw (8,0.15) node[above] {$32$};
\draw (9,0.15) node[above] {$36$};
\draw (-0.1,0.5) node[left] {$1$};
\draw (-0.1,0) node[left] {$0$};
\draw (-0.1,-7) node[left] {$-14$};
\draw (0,-7.25)--(0,0.5);
\draw (0,0)--(9,0);
\draw (0.25,0)--(0.50000,-0.048862)--(0.75000,-0.12439)--(1.0000,-0.21716)--(1.2500,-0.32220)--(1.5000,-0.43653)
--(1.7500,-0.55820)--(2.0000,-0.68586)--(2.2500,-0.81852)--(2.5000,-0.95543)--(2.7500,-1.0960)--(3.0000,-1.2398)
--(3.2500,-1.3864)--(3.5000,-1.5356)--(3.7500,-1.6871)--(4.0000,-1.8406)--(4.2500,-1.9960)--(4.5000,-2.1531)
--(4.7500,-2.3118)--(5.0000,-2.4719)--(5.2500,-2.6334)--(5.5000,-2.7962)--(5.7500,-2.9602)--(6.0000,-3.1252)
--(6.2500,-3.2913)--(6.5000,-3.4584)--(6.7500,-3.6264)--(7.0000,-3.7952)--(7.2500,-3.9649)--(7.5000,-4.1354)
--(7.7500,-4.3066)--(8.0000,-4.4785)--(8.2500,-4.6510)--(8.5000,-4.8243)--(8.7500,-4.9981)--(9.0000,-5.1725);
\draw (0.25,0)--(0.50000,-0.048853)--(0.75000,-0.15022)--(1.0000,-0.24156)--(1.2500,-0.38257)--(1.5000,-0.49315)
--(1.7500,-0.60989)--(2.0000,-0.68586)--(2.2500,-0.96289)--(2.5000,-1.1532)--(2.7500,-1.3572)--(3.0000,-1.5033)
--(3.2500,-1.7666)--(3.5000,-1.9170)--(3.7500,-2.0415)--(4.0000,-2.1097)--(4.2500,-2.3659)--(4.5000,-2.5442)
--(4.7500,-2.7458)--(5.0000,-2.8428)--(5.2500,-3.0026)--(5.5000,-3.0056)--(5.7500,-3.1316)--(6.0000,-3.1252)
--(6.2500,-3.6488)--(6.5000,-4.1100)--(6.7500,-4.3777)--(7.0000,-4.5825)--(7.2500,-5.1425)--(7.5000,-5.3642)
--(7.7500,-5.6721)--(8.0000,-5.7070)--(8.2500,-6.1973)--(8.5000,-6.6231)--(8.7500,-6.9354)--(9.0000,-7.1499);
\end{tikzpicture}
\end{center}
\caption{The logarithm of sphere packing density as a function of dimension.  The upper curve is the numerically
optimized
linear programming bound, while the lower curve is the best packing currently known.  The truth lies somewhere in between.}
\label{fig:LPgraph}
\end{figure}
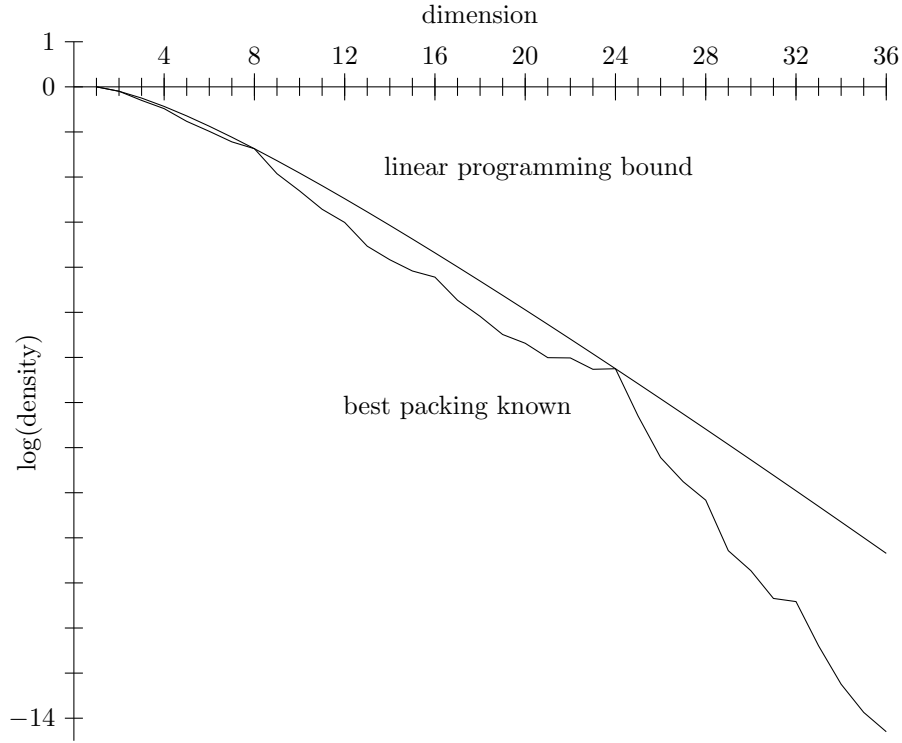

\begin{table}
\caption{The linear programming bound for the sphere packing density in $\R^n$ with $1 \le n \le 36$.  All numbers are rounded up.} \label{table:lpbound}
\begin{center}
\begin{tabular}{cccccccc}
\toprule
$n$ & upper bound & & $n$ & upper bound & & $n$ & upper bound\\
\cmidrule{1-2} \cmidrule{4-5} \cmidrule{7-8}
$1$ & $1.000000000$ & & $13$ & $0.0624817002$ & & $25$ & $0.001384190723$\\
$2$ & $0.906899683$ & & $14$ & $0.0463644893$ & & $26$ & $0.000991023890$\\
$3$ & $0.779746762$ & & $15$ & $0.0342482621$ & & $27$ & $0.000708229796$\\
$4$ & $0.647704966$ & & $16$ & $0.0251941308$ & & $28$ & $0.000505254217$\\
$5$ & $0.524980022$ & & $17$ & $0.0184640904$ & & $29$ & $0.000359858186$\\
$6$ & $0.417673416$ & & $18$ & $0.0134853405$ & & $30$ & $0.000255902875$\\
$7$ & $0.327455611$ & & $19$ & $0.0098179552$ & & $31$ & $0.000181708382$\\
$8$ & $0.253669508$ & & $20$ & $0.0071270537$ & & $32$ & $0.000128843289$\\
$9$ & $0.194555339$ & & $21$ & $0.0051596604$ & & $33$ & $0.000091235604$\\
$10$ & $0.147953479$ & & $22$ & $0.0037259420$ & & $34$ & $0.000064522197$\\
$11$ & $0.111690766$ & & $23$ & $0.0026842799$ & & $35$ & $0.000045574385$\\
$12$ & $0.083775831$ & & $24$ & $0.0019295744$ & & $36$ & $0.000032153056$\\
\bottomrule
\end{tabular}
\end{center}
\end{table}

The existence of these magic functions was conjectured by Cohn and Elkies
\cite{CohnElkies2003} on the basis of numerical evidence and analogies with
other problems in coding theory.  Further evidence was obtained by Cohn and
Kumar \cite{CohnKumar2009} in the course of proving that the Leech lattice is
the densest lattice in $\R^{24}$, while Cohn and Miller \cite{CohnMiller2016}
carried out an even more detailed study of the magic functions.  These
calculations left no doubt that the magic functions existed: one could
compute them to fifty decimal places, plot them, approximate their roots and
power series coefficients, etc. They were perfectly concrete and accessible
functions, amenable to exploration and experimentation, which indeed
uncovered various intriguing patterns. All that was missing was an existence
proof.

However, proving existence was no easy matter.  There was no sign of an
explicit formula, or any other characterization that could lead to a proof.
Instead, the magic functions seemed to come out of nowhere.

The fundamental difficulty is explaining where the magic comes from. One can
optimize the auxiliary function in any dimension, but that will generally not
produce a sharp bound for the packing density. Why should eight and
twenty-four dimensions be any different? The numerical results show that the
bound is nearly sharp in those dimensions, but why couldn't it be exact for a
hundred decimal places, followed by random noise?  That's not a plausible
scenario for anyone with faith in the beauty of mathematics, but faith does
not amount to a proof, and any proof must take advantage of special
properties of these dimensions.

For comparison, the answer is far less nice in sixteen dimensions. By analogy
with $r=\sqrt{2}$ when $n=8$ and $r=2$ when $n=24$, one might guess that
$r=\sqrt{3}$ when $n=16$, but that bound cannot be achieved.  Instead,
numerical optimization seems to converge to $r^2 = 3.0252593116828820\dots$,
which is close to $3$ but not equal to it.  This number has not yet been
identified exactly.

Despite the lack of an existence proof, the proof of
Theorem~\ref{thm:LPbounds} implicitly describes what the magic functions must
look like:

\begin{lemma}
Suppose $f$ satisfies the hypotheses of the linear programming bounds for
sphere packing in $\R^n$, with $f(x) \le 0$ for $|x| \ge r$, and suppose
$\Lambda$ is a lattice in $\R^n$ with minimal vector length $r$. Then the
density of $\Lambda$ equals the bound $\vol{\big(B_{r/2}^n\big)}$ from
Theorem~\ref{thm:LPbounds} if and only if $f$ vanishes on
$\Lambda\setminus\{0\}$ and $\widehat{f}$ vanishes on
$\Lambda^*\setminus\{0\}$.
\end{lemma}

\begin{proof}
Recall that the proof of Theorem~\ref{thm:LPbounds} for a lattice $\Lambda$
amounted to dropping all the nontrivial terms in the Poisson summation
formula, to obtain the inequality
\[
f(0) \ge \sum_{x \in \Lambda} f(x) = \frac{1}{\vol{(\R^n/\Lambda)}} \sum_{y \in \Lambda^*} \widehat{f}(y)
\ge \frac{\widehat{f}(0)}{\vol{(\R^n/\Lambda)}}.
\]
The only way this argument could yield a sharp bound is if all the omitted
terms were already zero.  In other words, $f$ proves that $\Lambda$ is an
optimal sphere packing if and only if $f$ vanishes on $\Lambda\setminus\{0\}$
and $\widehat{f}$ vanishes on $\Lambda^*\setminus\{0\}$.
\end{proof}

As discussed in the previous section, without loss of generality we can
assume that $f$ is a radial function, as is $\widehat{f}$. We know exactly
where the roots of $f$ and $\widehat{f}$ should be, since $E_8=E_8^*$ with
vector lengths $\sqrt{2k}$ for $k=1,2,\dots$, while
$\Lambda_{24}=\Lambda_{24}^*$ with vector lengths $\sqrt{2k}$ for
$k=2,3,\dots$.  These roots should have order two, to avoid sign changes,
except that the first root of $f$ should be a single root.  See
Figure~\ref{fig:schematic} for a diagram.

\begin{figure}
\begin{center}
\begin{tikzpicture}[scale=1.2]
\draw (0,-0.75) -- (0,1.75);
\draw (0,0) -- (4.5,0);
\draw (2.25,1) node {$f$};
\draw (0,1.5) to[out=0,in=106] (1,0) to[out=286,in=180] (1.4,-0.5)
to[out=0,in=180] (2,0) to[out=0,in=180] (2.5,-0.25)
to[out=0,in=180] (3,0) to[out=0,in=180] (3.5,-0.125)
to[out=0,in=180] (4,0) to[out=0,in=180] (4.5,-0.0625);
\draw (1,-0.1) -- (1,0.1); \draw (0.75,-0.1) node[below] {$\sqrt{2}$};
\draw (2,-0.1) -- (2,0.1); \draw (2,-0.1) node[below] {$\sqrt{4}$};
\draw (3,-0.1) -- (3,0.1); \draw (3,-0.1) node[below] {$\sqrt{6}$};
\draw (4,-0.1) -- (4,0.1); \draw (4,-0.1) node[below] {$\sqrt{8}$};
\end{tikzpicture}
\hskip 1cm
\begin{tikzpicture}[scale=1.2]
\draw (0,-0.75) -- (0,1.75);
\draw (0,0) -- (4.5,0);
\draw (2.25,1) node {$\widehat{f}$};
\draw (0,1.5) to[out=0,in=180] (1,0) to[out=0,in=180] (1.5,0.4)
to[out=0,in=180] (2,0) to[out=0,in=180] (2.5,0.2)
to[out=0,in=180] (3,0) to[out=0,in=180] (3.5,0.1)
to[out=0,in=180] (4,0) to[out=0,in=180] (4.5,0.05);
\draw (1,-0.1) -- (1,0.1); \draw (1,-0.1) node[below] {$\sqrt{2}$};
\draw (2,-0.1) -- (2,0.1); \draw (2,-0.1) node[below] {$\sqrt{4}$};
\draw (3,-0.1) -- (3,0.1); \draw (3,-0.1) node[below] {$\sqrt{6}$};
\draw (4,-0.1) -- (4,0.1); \draw (4,-0.1) node[below] {$\sqrt{8}$};
\end{tikzpicture}
\end{center}
\caption{A schematic diagram showing the roots of the magic
function $f$ and its Fourier transform $\widehat{f}$ in eight dimensions.
The figure is not to scale, because the actual functions
decrease too rapidly for an accurate plot to be illuminating.}
\label{fig:schematic}
\end{figure}
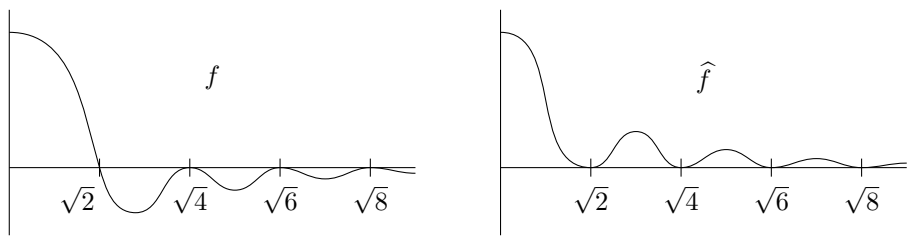

Thus, our problem is simple to state: how can we construct a radial Schwartz
function $f$ such that $f$ and $\widehat{f}$ have the desired roots and no
others?  Note that Poisson summation over $E_8$ or $\Lambda_{24}$ then
implies that $f(0) = \widehat{f}(0)$, and flipping the sign of $f$ if
necessary ensures that all the necessary inequalities hold.

Unfortunately it's difficult to take advantage of this characterization. The
problem is that it's hard to control a function and its Fourier transform
simultaneously: it's easy to produce the desired roots in either one
separately, but not at the same time.  Our inability to control $f$ without
losing control of $\widehat{f}$ is at the root of the Heisenberg uncertainty
principle, and it's a truly fundamental obstacle.

One natural way to approach this problem is to carry out numerical
experiments. Cohn and Miller used functions of the form $f(x) = p(|x|^2)
e^{-\pi |x|^2}$ to approximate the magic functions, where $p$ is a polynomial
chosen to force $f$ and $\widehat{f}$ to have many of the desired roots. Such
an approximation can never be exact, since it has only finitely many roots,
but it can come arbitrarily close to the truth. This investigation uncovered
several noteworthy properties of the magic functions, which showed that they
had unexpected structure. For example, if we normalize the magic functions
$f_8$ and $f_{24}$ in $8$ and $24$ dimensions so that $f_8(0)=f_{24}(0)=1$,
then Cohn and Miller conjectured that their second Taylor coefficients are
rational:
\begin{align*}
f_8(x) &= 1 - \frac{27}{10}|x|^2 + O\big(|x|^4\big), &\quad \widehat{f}_8(x) &= 1 - \frac{3}{2}|x|^2 + O\big(|x|^4\big),\\
f_{24}(x) &= 1 - \frac{14347}{5460}|x|^2 + O\big(|x|^4\big), &\quad \widehat{f}_{24}(x) &= 1 - \frac{205}{156}|x|^2 + O\big(|x|^4\big).
\end{align*}
If all the higher-order coefficients had been rational as well, then it would
have opened the door to determining these functions exactly, but
frustratingly it seems that the other coefficients are far more subtle and
presumably irrational. The magic functions retained their mystery, and this
Taylor series behavior went unexplained until the exact formulas for the
magic functions were discovered.

Given the difficulty of controlling $f$ and $\widehat{f}$ simultaneously, one
natural approach is to split them into eigenfunctions of the Fourier
transform. By Fourier inversion, every radial function $f$ satisfies
$\widehat{\widehat{f\,}}=f$. % ad hoc space
Thus, if we set $f_+ = \big(f+\widehat{f}\,\big)/2$ % ad hoc space
and $f_- = \big(f-\widehat{f}\,\big)/2$, % ad hoc space
then $f=f_++f_-$ with $\widehat{f}_+=f_+$ and $\widehat{f}_-=-f_-$.  Because
$f$ and $\widehat{f}$ vanish at the same points, they share these roots with
$f_+$ and $f_-$.  Our goal is therefore to construct radial eigenfunctions of
the Fourier transform with prescribed roots.  The advantage of this approach
is that it conveniently separates into two distinct problems, namely
constructing the $+1$ and $-1$ eigenfunctions, but these problems remain
difficult.

\section{Modular forms}
\label{sec:modular}

Ever since the Cohn-Elkies paper in 2003, number theorists had hoped to
construct the magic functions using modular forms. The reasoning is simple:
modular forms are deep and mysterious functions connected with lattices, as
are the magic functions, so wouldn't it make sense for them to be related?
Unfortunately, they are entirely different sorts of functions, with no clear
connection between them.  That's where matters stood until Viazovska
discovered a remarkable integral transform, which enabled her to construct
the magic functions using modular forms. We'll get there shortly, but first
let's briefly review how modular forms work.

We'll start with some examples.  Every lattice $\Lambda$ has a theta series
$\Theta_\Lambda$, defined by
\begin{equation} \label{eq:thetadef}
\Theta_\Lambda(z) = \sum_{x \in \Lambda} e^{\pi i |x|^2 z}.
\end{equation}
This series converges when $\Im z > 0$, and it defines an analytic function
on the upper half-plane $\UHP = \{z \in \C : \Im z > 0\}$.  To motivate the
definition, think of the theta series as a generating function, where the
coefficient of $e^{\pi i t z}$ counts the number of $x \in \Lambda$ with
$|x|^2=t$.  However, there's one aspect not explained by the generating
function interpretation: why write this function in terms of $e^{\pi i z}$?
Doing so may at first look like a gratuitous nod to Fourier series, but it
leads to an elegant transformation law based on applying Poisson summation to
a Gaussian:

\begin{proposition}
If $\Lambda$ is a lattice in $\R^n$, then
\[
\Theta_\Lambda(z) = \frac{1}{\vol{(\R^n/\Lambda)}}\left(\frac{i}{z}\right)^{n/2} \Theta_{\Lambda^*}(-1/z)
\]
for all $z \in \UHP$.
\end{proposition}

\begin{proof}
One of the most important properties of Gaussians is that the set of
Gaussians is closed under the Fourier transform: the Fourier transform of a
wide Gaussian is a narrow Gaussian, and vice versa.  More precisely, for
$t>0$ the Fourier transform of the Gaussian $x \mapsto e^{-t \pi |x|^2}$ on
$\R^n$ is $x \mapsto t^{-n/2} e^{-\pi |x|^2/t}$. In fact, the same holds
whenever $t$ is a complex number with $\Re t > 0$, by analytic continuation.
Then Poisson summation tells us that
\[
\sum_{x \in \Lambda} e^{-t \pi |x|^2} =
\frac{1}{\vol{(\R^n/\Lambda)}} \sum_{y \in \Lambda^*}  t^{-n/2} e^{-\pi |y|^2/t}.
\]
Setting $z = i t$, we find that
\[
\Theta_\Lambda(z) = \frac{1}{\vol{(\R^n/\Lambda)}}\left(\frac{i}{z}\right)^{n/2} \Theta_{\Lambda^*}(-1/z)
\]
whenever $\Im z > 0$, as desired.
\end{proof}

If we set $\Lambda = E_8$, then $\Lambda^* = E_8$ as well, and we find that
\[
\Theta_{E_8}(-1/z) = z^4 \Theta_{E_8}(z).
\]
Furthermore, $E_8$ is an even lattice, and hence the Fourier series
\eqref{eq:thetadef} implies that
\[
\Theta_{E_8}(z+1) = \Theta_{E_8}(z).
\]
These two symmetries are the most important properties of $\Theta_{E_8}$. For
exactly the same reasons, the theta series of the Leech lattice
$\Lambda_{24}$ satisfies
\[
\Theta_{\Lambda_{24}}(-1/z) = z^{12} \Theta_{\Lambda_{24}}(z) \qquad\text{and}\qquad
\Theta_{\Lambda_{24}}(z+1) = \Theta_{\Lambda_{24}}(z).
\]

The mappings $z \mapsto z+1$ and $z \mapsto -1/z$ generate a discrete group
of transformations of the upper half-plane, called the \emph{modular group}.
It turns out to be the same as the action of the group $\SL_2(\Z)$ on the
upper half-plane by linear fractional transformations, but we will not need
this fact except for naming purposes.

A \emph{modular form} of weight $k$ for $\SL_2(\Z)$ is a holomorphic function
$\varphi \colon \UHP \to \C$ such that $\varphi(z+1) = \varphi(z)$ and
$\varphi(-1/z) = z^k \varphi(z)$ for all $z \in \UHP$, while $\varphi(z)$
remains bounded as $\Im z \to \infty$. (The latter condition is called being
\emph{holomorphic at infinity}, because it means the singularity there is
removable.)  It's not hard to show that the weight of a nonzero modular form
must be nonnegative and even, and the only modular forms of weight zero are
the constant functions.

We have seen that $\Theta_{E_8}$ and $\Theta_{\Lambda_{24}}$ satisfy the
transformation laws for modular forms of weight $4$ and $12$, respectively,
and it is easy to check that they are holomorphic at infinity.  Thus, these
theta series are modular forms.

There are a number of other well-known modular forms.  For example, the
\emph{Eisenstein series} $E_k$ defined by
\[
E_k(z) = \frac{1}{2\zeta(k)}\sum_{\substack{(m,n)\in\Z^2\\(m,n)\ne (0,0)}}\frac{1}{(mz+n)^{k}}
\]
is a modular form of weight $k$ for $\SL_2(\Z)$ whenever $k$ is an even
integer greater than $2$ (while it vanishes when $k$ is odd).  The proofs of
the required identities $E_k(z+1)=E_k(z)$ and $E_k(-1/z) = z^k E_k(z)$ simply
amount to rearranging the sum. Here $\zeta$ denotes the Riemann zeta
function, and $2\zeta(k)$ is a normalizing factor. The advantage of this
normalization is that it leads to the Fourier expansion
\begin{equation} \label{eq:eisensteinfourier}
E_k(z) = 1 + \frac{2}{\zeta(1-k)} \sum_{m=1}^\infty \sigma_{k-1}(m) e^{2\pi i m z},
\end{equation}
where $\sigma_{k-1}(m)$ is the sum of the $(k-1)$-st powers of the divisors
of $m$ and $\zeta(1-k)$ turns out to be a rational number.

The notational conflict between the Eisenstein series $E_k$ and the $E_8$
lattice is unfortunate, but both notations are well established. Fortunately,
we will never need to set $k=8$, and the context should easily distinguish
between Eisenstein series and lattices.

Modular forms are highly constrained objects, which makes coincidences
commonplace.  For example, $\Theta_{E_8}$ is the same as $E_4$, because there
is a unique modular form of weight $4$ for $\SL_2(\Z)$ with constant term
$1$.  Equivalently, for $m=1,2,\dots$ there are exactly $240\sigma_3(m)$
vectors $x \in E_8$ with $|x|^2=2m$. The theta series $\Theta_{\Lambda_{24}}$
is not an Eisenstein series, but it can be written in terms of them as
\[
\Theta_{\Lambda_{24}} = \frac{7}{12}E_4^3+\frac{5}{12}E_6^2.
\]
More generally, let $\cM_k$ denote the space of modular forms of weight $k$
for $\SL_2(\Z)$. Then $\bigoplus_{k \ge 0} \cM_k$ is a graded ring, because
the product of modular forms of weights $k$ and $\ell$ is a modular form of
weight $k+\ell$. This ring is isomorphic to a polynomial ring on two
generators, namely $E_4$ and $E_6$.  In other words, the set
\[
\left\{E_4^i E_6^j : \text{$i,j \ge 0$ and $4i+6j = k$} \right\}
\]
is a basis for the modular forms of weight $k$.  In particular, there is no
modular form of weight $2$ for $\SL_2(\Z)$, because the weights of $E_4$ and
$E_6$ are too high to generate such a form.

One cannot obtain a modular form of weight $2$ by setting $k=2$ in the double
sum definition of $E_k$.  The problem is that rearranging the terms is
crucial for proving modularity, but when $k=2$ the series converges only
conditionally, not absolutely.  Instead, we can define $E_2$ using
\eqref{eq:eisensteinfourier}. That defines a merely quasimodular form, rather
than an actual modular form, because one can show that $E_2(-1/z) = z^2
E_2(z) - 6iz/\pi$ rather than $z^2 E_2(z)$. This imperfect Eisenstein series
will play a role in constructing the magic functions.

By default all modular forms are required to be holomorphic, but we can of
course consider quotients that are no longer holomorphic.  A
\emph{meromorphic modular form} is the quotient of two modular forms, and it
is \emph{weakly holomorphic} if it is holomorphic on $\UHP$ (but not
necessarily at infinity).  Unlike the holomorphic case, there is an
infinite-dimensional space of weakly holomorphic modular forms of each even
weight, positive or negative.  Allowing a pole at infinity offers tremendous
flexibility.

On the face of it, modular forms seem to have little to do with the magic
functions.  In particular, it's not clear what modular forms have to do with
the radial Fourier transform in $n$ dimensions.  One hint that they may be
relevant comes from the Laplace transform.  As we saw when we looked at theta
series, Gaussians are a particularly useful family of functions for which we
can easily compute the Fourier transform.  It's natural to define a function
$f$ as a continuous linear combination of Gaussians via
\[
f(x) = \int_0^\infty e^{-t \pi |x|^2} g(t) \, dt,
\]
where the weighting function $g(t)$ gives the coefficient of the Gaussian
$e^{-t \pi |x|^2}$.  This formula is simply the \emph{Laplace transform} of
$g$, evaluated at $\pi|x|^2$.

Assuming $g$ is sufficiently well behaved, we can compute $\widehat{f}$ by
interchanging the Fourier transform with the integral over $t$, which yields
\begin{align*}
\widehat{f}(y) &= \int_0^\infty t^{-n/2} e^{-\pi |y|^2/t} g(t) \, dt\\
&= \int_0^\infty e^{-t \pi |y|^2} t^{n/2-2} g(1/t) \, dt.
\end{align*}
In other words, taking the Fourier transform of $f$ amounts to replacing $g$
with $t \mapsto t^{n/2-2} g(1/t)$.

As a consequence, if $g(1/t) = \varepsilon t^{2-n/2} g(t)$ with $\varepsilon
= \pm 1$, then $\widehat{f} = \varepsilon f$.  Thus, we can construct
eigenfunctions of the Fourier transform by taking the Laplace transform of
functions satisfying a certain functional equation. What's noteworthy about
this functional equation is how much it looks like the transformation law for
a modular form on the imaginary axis.  If we set $g(t) = \varphi(it)$, then
the modular form equation $\varphi(-1/z) = z^k \varphi(z)$ with $z=it$
corresponds to $g(1/t) = i^k t^k g(t)$. If $\varphi$ is a meromorphic modular
form of weight $k=2-n/2$ that vanishes at $i\infty$ and has no poles on the
imaginary axis, then $f$ is a radial eigenfunction of the Fourier transform
in $\R^n$ with eigenvalue $i^k$.

Of course this is far from the only way to construct Fourier eigenfunctions,
but it's a natural way to construct them from modular forms.  As stated here,
it's clearly not flexible enough to construct the magic functions, because it
produces only one eigenvalue.  If we take $n=8$ and weight $k=2-n/2 = -2$,
then $i^k = -1$, so we can construct a $-1$ eigenfunction but not a $+1$
eigenfunction for the same dimension.  This turns out not to be a serious
obstacle: there are many variants of modular forms (for other groups or with
characters), and it's not hard to produce eigenfunctions with both
eigenvalues.  However, there's a much worse problem. If we build an
eigenfunction this way, then there's no obvious way to control the roots of
the eigenfunction using the Laplace transform.  Given that our goal is to
prescribe the roots, this approach seems to be useless.  What's holding us
back is that we have not taken full advantage of the modular form: we are
using only the identity $\varphi(-1/z) = z^k \varphi(z)$, and not
$\varphi(z+1)=\varphi(z)$.

\section{Viazovska's proof}

The fundamental problem with the Laplace transform approach in the previous
section is that it seems to be impossible to achieve the desired roots.
Viazovska gets around this difficulty by a bold construction: she simply
inserts the desired roots by brute force, by including an explicit factor of
$\sin^2\mathopen{}\big(\pi |x|^2/2\big)\mathclose{}$, which vanishes to
second order at $|x| = \sqrt{2k}$ for $k=1,2,\dots$ and fourth order at
$x=0$. In her construction for eight dimensions, both eigenfunctions have the
form
\begin{equation} \label{eq:viazovska-form}
\sin^2\mathopen{}\big(\pi |x|^2/2\big)\mathclose{} \int_0^\infty g(t) e^{-\pi |x|^2 t} \, dt
\end{equation}
for some function $g$.

One obvious issue with this approach is that $\sin^2\mathopen{}\big(\pi
|x|^2/2\big)\mathclose{}$ vanishes more often than we would like.
Specifically, it vanishes to fourth order when $x=0$ and second order when
$|x|=\sqrt{2}$, whereas we wish to have no root when $x=0$ and only a
first-order root when $|x|=\sqrt{2}$.  To avoid this difficulty, the integral
in \eqref{eq:viazovska-form} must have poles at $0$ and $\sqrt{2}$ as a
function of $|x|$, which cancel the unwanted roots. The integral will
converge only for $|x|>\sqrt{2}$, but the function defined by
\eqref{eq:viazovska-form} extends to $|x| \le \sqrt{2}$ by analytic
continuation.

Which choices of $g$ will produce eigenfunctions of the Fourier transform in
$\R^8$? This is not clear, because the factor of $\sin^2\mathopen{}\big(\pi
|x|^2/2\big)\mathclose{}$ disrupts the straightforward Laplace transform
calculations from the end of Section~\ref{sec:modular}. Instead, Viazovska
writes the sine function in terms of complex exponentials and carries out
elegant contour integral arguments to show that \eqref{eq:viazovska-form}
gives an eigenfunction whenever $g$ satisfies certain transformation laws.
Identifying the right conditions on $g$ is not at all obvious, and it's the
heart of her paper.

To get a $+1$ eigenfunction, Viazovska shows that it suffices to take $g(t) =
t^2 \varphi(i/t)$, where $\varphi$ is a weakly holomorphic quasimodular form
of weight $0$ and depth $2$ for $\SL_2(\Z)$.  Here, a \emph{quasimodular form
of depth $2$} is a quadratic polynomial in $E_2$ with modular forms as
coefficients, where $E_2$ is the Eisenstein series of weight $2$. Recall that
$E_2$ fails to be a modular form because of the strange transformation law
$E_2(-1/z) = z^2 E_2(z) - 6iz/\pi$, but that functional equation works
perfectly here.

To get a $-1$ eigenfunction, Viazovska shows that it suffices to take $g(t) =
\psi(it)$, where $\psi$ is a weakly holomorphic modular form of weight $-2$
for a subgroup of $\SL_2(\Z)$ called $\Gamma(2)$ and $\psi$ satisfies the
additional functional equation
\[
\psi(z) = \psi(z+1) + z^2\psi(-1/z).
\]
We have not discussed modular forms for other groups such as $\Gamma(2)$, but
they are similar in spirit to those for $\SL_2(\Z)$.  In particular, the ring
of modular forms for $\Gamma(2)$ is generated by two forms of weight $2$,
namely $\Theta_\Z^4$ (the fourth power of the theta series of the
one-dimensional integer lattice) and its translate $z\mapsto
\Theta_\Z^4(z+1)$.

These conditions for $\varphi$ and $\psi$ are every bit as arcane as they
look. It's far from obvious that they lead to eigenfunctions, but Viazovska's
contour integral proof shows that they do.  Even once we know that this
method gives eigenfunctions, it's unclear how to choose $\varphi$ and $\psi$
to yield the magic eigenfunctions, or whether this is possible at all.

Fortunately, one can write down some necessary conditions, and then the
simplest functions satisfying those conditions work perfectly. In particular,
we can take
\[
\varphi = \frac{4\pi(E_2E_4-E_6)^2}{5(E_6^2-E_4^3)}
\]
and
\[
\psi = -\frac{32\Theta_\Z^4|_T\big(5\Theta_\Z^8 - 5 \Theta_\Z^4|_T \Theta_\Z^4  + 2 \Theta_\Z^8|_T\big)}
{15 \pi \Theta_\Z^8\big(\Theta_\Z^4 - \Theta_\Z^4|_T\big)^2},
\]
where $f|_T$ denotes the translate $z \mapsto f(z+1)$ of a function $f$.

Thus, to obtain the magic function for $E_8$ we set
\begin{equation} \label{eq:magic-formula}
f(x) = \sin^2\mathopen{}\big(\pi |x|^2/2\big)\mathclose{}
\int_0^\infty \big(t^2 \varphi(i/t) + \psi(it)\big) e^{-\pi |x|^2 t} \, dt
\end{equation}
for the specific $\varphi$ and $\psi$ identified by Viazovska. Because the
$\varphi$ and $\psi$ terms yield eigenfunctions of the Fourier transform, we
find that
\[
\widehat{f}(y) = \sin^2\mathopen{}\big(\pi |y|^2/2\big)\mathclose{}
\int_0^\infty \big(t^2 \varphi(i/t) - \psi(it)\big) e^{-\pi |y|^2 t} \, dt.
\]
The integral in the formula for $f(x)$ converges only when $|x|>\sqrt{2}$,
but the one in the formula for $\widehat{f}(y)$ turns out to converge
whenever $|y|>0$, because the problematic growth of the integrand cancels in
the difference $t^2 \varphi(i/t) - \psi(it)$.

These formulas define Schwartz functions that have the desired roots, and one
can check that $f(0) = \widehat{f}(0)=1$, but it's not obvious that they
satisfy the inequalities $f(x) \le 0$ for $|x| \ge \sqrt{2}$ and
$\widehat{f}(y) \ge 0$ for all $y$, because there might be additional sign
changes. In fact, these inequalities hold for a fundamental reason:
\begin{equation} \label{eq:ineqs}
t^2 \varphi(i/t) + \psi(it) < 0 \qquad\text{and}\qquad t^2 \varphi(i/t) - \psi(it) > 0
\end{equation}
for all $t \in (0,\infty)$.  In other words, the inequalities already hold at
the level of the quasimodular forms, with no need to worry about the Laplace
transform except to observe that it preserves positivity.  Note that the
restriction of the inequality $f(x) \le 0$ to $|x| \ge \sqrt{2}$ fits
perfectly into this framework, because the integral in
\eqref{eq:magic-formula} diverges for $|x| < \sqrt{2}$ and thus we do not
obtain $f(x) \le 0$ there. All that remains is to prove the inequalities
\eqref{eq:ineqs}. Unfortunately, no simple proof of these inequalities is
known at present, but one can verify them by reducing the problem to a finite
calculation.

Thus, Viazovska's formula \eqref{eq:magic-formula} defines the long-sought
magic function for $E_8$ and solves the sphere packing problem in eight
dimensions. What about twenty-four dimensions?  The same basic approach
works, but choosing the quasimodular forms requires more effort. Fortunately,
the conjectures by Cohn and Miller can be used to help pin down the right
choices. Once the magic function has been identified, there are additional
technicalities involved in verifying the inequality for $\widehat{f}$, but
these challenges can be overcome, which leads to a solution of the sphere
packing problem in twenty-four dimensions.

\section{Future prospects}

Nobody expects Viazovska's proof to generalize to any other dimensions above
two. Why just eight and twenty-four?  At one level, we really don't know why.
Nobody has been able to find a proof, or even a compelling heuristic
argument, that rules out similar phenomena in higher dimensions. We can't
even rule out the possibility that linear programming bounds might solve the
sphere packing problem in every sufficiently high dimension, although that's
clearly ridiculous.

Despite our lack of understanding, the special role of eight and twenty-four
dimensions aligns with our experience elsewhere in mathematics. Mathematics
is full of exceptional or sporadic phenomena that occur in only finitely many
cases, and the $E_8$ and Leech lattices are prototypical examples. These
objects do not occur in isolation, but rather in constellations of remarkable
structures. For example, both $E_8$ and the Leech lattice are connected with
binary error-correcting codes, combinatorial designs, spherical designs,
finite simple groups, etc.  Each of these connections constrains the
possibilities, especially given the classification of finite simple groups,
and there just doesn't seem to be room for a similar constellation in higher
dimensions.

Instead, solving the sphere packing problem in further dimensions will
presumably require new techniques.  One particularly attractive case is the
$D_4$ root lattice, which is surely the best sphere packing in $\R^4$.  This
lattice shares some of the wonderful properties of $E_8$ and the Leech
lattice, but not enough for the four-dimensional linear programming bound to
be sharp. It would be a plausible target for any generalization of this
bound, and in fact such a generalization may be emerging.

Building on work of Schrijver, Bachoc and Vallentin, and other researchers,
de~Laat and Vallentin have generalized linear programming bounds to a
hierarchy of semidefinite programming bounds \cite{dLV2015}.  Linear
programming bounds are the first level of this hierarchy, which means that
$E_8$ and the Leech lattice have the simplest possible proofs from this
perspective.  What about $D_4$? Perhaps this case can be solved at one of the
next few levels of the hierarchy. Much work remains to be done here, and it's
unclear what the prospects are for any particular dimension, but it is not
beyond hope that four dimensions could someday join eight and twenty-four
among the solved cases of the sphere packing problem.

\section*{Acknowledgments}

I am grateful to James Bernhard, Donald Cohn, Matthew de Courcy-Ireland,
Stephen D.\ Miller, Frank Morgan, David Rohrlich, Achill Sch\"urmann, Frank
Vallentin, and Maryna Viazovska for their feedback and suggestions.

\section*{Photo Credits}

\noindent Figure~\ref{fig:maryna} is courtesy of Daniil Yevtushynsky.

\noindent The photos in Figure~\ref{fig:leech} are courtesy of Mary Caisley,
Mark Ostow, C.~J.~Mozzochi, and Julia Semikina, from left to right.

\noindent Figures~\ref{fig:low-dimensions} and~\ref{fig:noam} are courtesy of
Henry Cohn.

\noindent Figure~\ref{fig:steve-class} is courtesy of Matthew Kownacki.


\begin{thebibliography}{99}
\bibitem{Cohn2016} H.~Cohn, \emph{Packing, coding, and ground states},
    PCMI 2014 lecture notes, 2016.
    \arXiv{1603.05202}

\bibitem{CohnElkies2003} H.~Cohn and N.~Elkies, \emph{New upper bounds on
    sphere packings I}, Ann.\ of Math.\ (2) \textbf{157} (2003), no.\ 2, 689--714. \arXiv{math/0110009} \MR{1973059}
    \doi{10.4007/annals.2003.157.689}

\bibitem{CohnKumar2009} H.~Cohn and A.~Kumar, \emph{Optimality and
    uniqueness of the Leech lattice among lattices},
    Ann.\ of Math.\ (2) \textbf{170} (2009), no.\ 3, 1003--1050.
    \arXiv{math/0403263} \MR{2600869} \doi{10.4007/annals.2009.170.1003}

\bibitem{CKMRV24} H.~Cohn, A.~Kumar, S.~D.~Miller, D.~Radchenko, and
    M.~Viazovska, \emph{The sphere packing problem in dimension $24$}, preprint, 2016. \arXiv{1603.06518}

\bibitem{CohnMiller2016} H.~Cohn and S.~D.~Miller, \emph{Some
    properties of optimal functions for sphere packing in dimensions $8$
    and $24$},
    preprint, 2016. \arXiv{1603.04759}

\bibitem{ConwaySloane1995} J.~H.~Conway and N.~J.~A.~Sloane, \emph{What are
    all the best sphere packings in low dimensions?},
    Discrete Comput.\ Geom.\ \textbf{13} (1995), no.\ 3--4, 383--403.  \MR{1318784}
    \doi{10.1007/BF02574051}

\bibitem{SPLAG} J.~H.~Conway and N.~J.~A.~Sloane, \emph{Sphere
    packings, lattices and groups}, third edition,
    Grundlehren der Mathematischen Wissenschaften \textbf{290},
    Springer, New York, 1999. \MR{1662447} \doi{10.1007/978-1-4757-6568-7}

\bibitem{Hales2000} T.~C.~Hales, \emph{Cannonballs and honeycombs}, Notices
    Amer.\ Math.\ Soc.\ \textbf{47} (2000), no.\ 4, 440--449. \MR{1745624}

\bibitem{Hales2005} T.~C.~Hales, \emph{A proof of the Kepler
    conjecture},
    Ann.\ of Math.\ (2) \textbf{162} (2005), no.\ 3, 1065--1185. \MR{2179728} \doi{10.4007/annals.2005.162.1065}

\bibitem{FPK} T.~Hales, M.~Adams, G.~Bauer, D.~T.~Dang, J.~Harrison,
    T.~L.~Hoang, C.~Kaliszyk, V.~Magron, S.~McLaughlin, T.~T.~Nguyen,
    T.~Q.~Nguyen, T.~Nipkow, S.~Obua, J.~Pleso, J.~Rute,
    A.~Solovyev, A.~H.~T.~Ta, T.~N.~Tran, D.~T.~Trieu, J.~Urban, K.~K.~Vu, and
    R.~Zumkeller, \emph{A formal proof of the Kepler
    conjecture}, preprint, 2015. \arXiv{1501.02155}

\bibitem{KabatyanskiiLevenshtein1978} G.~A.~Kabatyanskii % spelled as in the English translation
    and V.~I.~Levenshtein, \emph{Bounds for packings on a sphere and in
    space}, Problems Inform.\ Transmission \textbf{14} (1978), no.\ 1, 1--17. \MR{0514023}

\bibitem{dLV2015} D.~de Laat and F.~Vallentin, \emph{A semidefinite
    programming hierarchy for packing problems in discrete geometry},
    Math.\ Program.\ \textbf{151} (2015), no.\ 2, Ser.\ B, 529--553. \arXiv{1311.3789}
    \MR{3348162} \doi{10.1007/s10107-014-0843-4}

\bibitem{dLV2016} D.~de Laat and F.~Vallentin,
    \emph{A breakthrough in sphere packing: the search for magic
    functions}, Nieuw Arch.\ Wiskd.\ (5) \textbf{17} (2016), no.\ 3, 184--192. \arXiv{1607.02111}
    % add MR and doi when available

\bibitem{Venkatesh2013} A.~Venkatesh, \emph{A note on sphere
    packings in high dimension}, Int.\ Math.\ Res.\ Not.\
    \textbf{2013} (2013), no.\ 7, 1628--1642. \MR{3044452} \doi{10.1093/imrn/rns096}

\bibitem{Viazovska2016} M.~S.~Viazovska, \emph{The sphere packing problem in
    dimension $8$}, preprint, 2016. \arXiv{1603.04246}
\end{thebibliography}
\end{document}